\documentclass[a4paper,10pt]{article}

\usepackage[english]{babel}
\usepackage{a4wide}
\usepackage{amssymb,amsfonts,amsmath,latexsym,dcolumn}
\usepackage{amsthm,mathdots,mathtools}
\usepackage{graphicx}
\usepackage{subfigure}
\usepackage{cite}
\usepackage{hyperref}

\newcommand{\email}[1]{#1}

\newtheorem{defi}{Definition}
\newtheorem{theorem}[defi]{Theorem}
\newtheorem{lemma}[defi]{Lemma}

\newtheorem{example}[defi]{Example}
\newtheorem{remark}[defi]{Remark}

\newcommand{\x}{\lambda}
\renewcommand{\leq}{\leqslant}

\usepackage{
  tikz,%
  pgfplots, %
  xcolor, %
}
\pgfplotsset{compat=1.8}
\usetikzlibrary{decorations.pathmorphing}
\usetikzlibrary{calc}
\usetikzlibrary{arrows}
\usetikzlibrary{fit}
\usetikzlibrary{backgrounds}
\usetikzlibrary{matrix}
\usetikzlibrary{decorations.pathreplacing} %
\usetikzlibrary{external}
\usepackage{rotfig}

 \DeclarePairedDelimiter{\norm}{\lVert}{\rVert}

 \definecolor{SPECorange}{rgb}{1.0,.5625,0}
 \definecolor{SPECblue}{rgb}{0,0,0.75}
 \definecolor{SPECred}{rgb}{0.75,0,0}
 \definecolor{SPECgreen}{rgb}{0,0.75,0}
 \definecolor{SPECblack}{rgb}{0.75,0.75,0.75}
 \definecolor{mainblue}{HTML}{225292}
 \definecolor{citecol}{rgb}{0.75,0,0}
 \definecolor{linkcol}{rgb}{0,0,0.75}
 \definecolor{SPECpink}{rgb}{1.0,0.1,0.6}
 \definecolor{SPECorangeg}{rgb}{0.609375,.609375,0.609375}
 \definecolor{SPECblueg}{rgb}{0.0525,0.0525,0.0525}

\newcommand{\CS}[1]{\mathcal{#1}}
\usepackage{xargs} 
\newcommandx{\tikzrotation}[6][1=black,2=,3=,4=]{
  \node at (#5,#6+.025) [align=center,scale=1.0,color=#1]{$\Rc$}; %
  \node at (#5,#6+.270) [align=center,scale=1.0,color=#1]{$\rc$}; %
  \node at (#5+.1,#6+.1) [align=center,scale=1.0,color=#1]{#2};%
  \node at (#5,#6+0.4) [align=center,scale=1.0,color=#1]{#3};%
  \node at (#5,#6-0.2) [align=center,scale=1.0,color=#1] {#4};
}

\newcommand{\uppertriangular}[4][black]{
  \draw[color=#1] (#2,#3) -- (#2+1.0,#3) -- (#2+1.0,#3+1.0) -- cycle;
  \node at (#2+.7,#3+.3) [align=center,scale=1.0,color=#1] {#4};%
}
\newcommand{\upperhessenberg}[4][black]{
  \draw[color=#1] (#2,#3) -- (#2+1.0,#3) -- (#2+1.0,#3+1.0) -- 
  (#2+.8,#3+1.0) -- (#2,#3+.2) -- cycle;
  \node at (#2+.6,#3+.35) [align=center,scale=1.0,color=#1] {#4};%
}

\newcommand{\shiftthroughrl}[4][black]{
  \path[color=#1,->,out=180,in=0,looseness=0.25] (#2-0.1,#4+0.1) edge
  (#3+0.05,#4+0.1); }

\newcommand{\shiftthroughlr}[4][black]{
  \path[color=#1,->,out=0,in=180,looseness=0.25] (#2+0.05,#4+0.1) edge
  (#3-0.1,#4+0.1); }

\newcommand{\transferbulgelr}[4][black]{ 
  \path[->,out=-45,in=-135,looseness=.25] (#2+0.1,#4+0.2) edge (#3-0.1,#4+0.2);
}

\newcommand{\turnoverrl}[5][black]{
  \path[color=#1,->,out=180,in=0] (#2-0.1,#3+0.1) edge (#4+0.05,#5+0.1);
}

\newcommand{\turnoverlr}[5][black]{
  \path[color=#1,->,out=0,in=180] (#2+0.05,#3+0.1) edge (#4-0.1,#5+0.1);
}

\newcommand{\drawbrace}[5][black]{
  \draw[color=#1,decorate,decoration=brace] (#3+0.05,#4)--(#2-0.05,#4);
  \node[color=#1] at (${0.5}*(#2+#3,#4+#4+0.6)$) [align=center] {#5};
}

\usetikzlibrary{calc}

\newcommand{\CC}{{\mathbb{C}}}

\newcommand{\absval}[1]{|{#1}|}

\title{Fast
  and backward stable computation of eigenvalues and eigenvectors of matrix polynomials\thanks{The research was partially supported by 
 the Research Council KU Leuven, project
 C14/16/056 (Invers-free Rational Krylov Methods: Theory and Applications)
, and by the GNCS/INdAM project ``Metodi numerici avanzati per equazioni e funzioni di matrici con struttura''.
}
}

\author{%
  Jared Aurentz\thanks{\email{Jared.Aurentz@icmat.es}, Instituto de Ciencias Matem\'aticas, Universidad Aut\'onoma de Madrid, Madrid, Spain},
  Thomas Mach\thanks{\email{Thomas.Mach@nu.edu.kz}, Department of Mathematics, Nazarbayev University, Kazakhstan},
  Leonardo Robol\thanks{\email{Leonardo.Robol@isti.cnr.it}, ISTI, Area della ricerca CNR, Pisa, Italy},
  Raf Vandebril\thanks{\email{Raf.Vandebril@cs.kuleuven.be}, Dept.\ of Computer Science, KU Leuven, Belgium},
  David S.~Watkins\thanks{\email{Watkins@math.wsu.edu}, Dept. of Mathematics, Washington State University, USA}
}

\date{\today}

\begin{document}

\maketitle


\begin{abstract}


  In the last decade matrix polynomials have been investigated
  with the primary focus on adequate linearizations and good
  scaling techniques for computing their eigenvalues and eigenvectors. In this
  article we propose a new method for computing a
  factored Schur form of the associated companion pencil.  The
  algorithm has a quadratic cost in the degree of the polynomial and a
  cubic one in the size of the coefficient matrices. Also the eigenvectors can
  be computed at the same cost.

  The algorithm is a variant of Francis's implicitly shifted QR algorithm applied on the
   companion pencil. A preprocessing unitary equivalence is executed on the
  matrix polynomial to simultaneously bring the leading matrix coefficient  and the
  constant matrix term to  triangular form before forming the companion pencil.  The
  resulting structure allows us to
  stably factor each matrix of the pencil as a product of  $k$ matrices of unitary-plus-rank-one form, 
  admitting cheap and numerically reliable storage. The problem is then solved as a
  product core chasing eigenvalue problem. A backward error analysis is included, implying
  normwise backward stability after a proper scaling. Computing the eigenvectors via
  reordering the Schur form is discussed as well.
 
Numerical experiments illustrate stability and efficiency of the proposed methods. \\[10pt]

\noindent\textbf{Keywords}: Matrix polynomial, product eigenvalue problem, core chasing algorithm,
eigenvalues, eigenvectors \\
\noindent\textbf{MSC class:} 65F15, 65L07

\end{abstract}




\section{Introduction}

  We are interested in computing  the eigenvalues and eigenvectors 
  of a degree $d$ square matrix polynomial:
  \begin{equation*} 
    P(\lambda) v = 0, \qquad\mbox{ where } \qquad P(\lambda) = \sum_{i = 0}^d P_i \lambda^i, \qquad 
    P_i \in \mathbb C^{k \times k} . 
  \end{equation*}
  The eigenpairs $(\lambda, v)$ are useful for a wide range of different applications,
  such as, for instance, the study of vibrations in structures
  \cite{q654,q624,lancaster2002lambda} and the numerical solution
  of differential equations \cite{nlevp},
  or the solution of certain matrix equations
  \cite{bini1996solution}.

  Even though recently a lot of interest has gone to studying other types of
  linearizations \cite{q759,DoLaPeVD16,q624}, the classical approach to solve this problem is still to construct 
  the so-called block companion pencil linearization 
  	\begin{equation} \label{eq:comp_pencil}
  	  S - \lambda T = \begin{bmatrix}
  	    \lambda I_k & && P_0   \\
  	    -I_k & \ddots &  & \vdots \\
  	    & \ddots & \lambda I_k & P_{d-2}  \\
  	    & & -I_k & \lambda P_d + P_{d-1}
  	  \end{bmatrix}\in\CC^{n\times n}[\lambda], 
  	\end{equation}
with $n=dk$ and $I_k$ the $k\times k$ identity, having eigenvalues identical to the eigenvalues of $P(\lambda)$ \cite{gohberg1982matrix,Ga74}.

 Even though the companion pencil is highly structured, the eigenvalues of $S - \lambda T$
 are usually computed by means of the QZ iteration \cite{q361}, which ignores the
 available structure. For example the MATLAB function {\tt polyeig} uses LAPACK's QZ
 implementation to compute the eigenvalues in this way. This approach has a cubic
 complexity in both the size of the matrices and the degree $\mathcal{O}(d^3k^3)$. The algorithm we propose is
 cubic in the size of the matrices, but quadratic in the degree $\mathcal{O}(d^2k^3)$. In all cases where
 high degree matrix polynomials are of interest, such as the
 approximation of the stationary vector for
 M/G/1 queues \cite{bini1996solution} or
 the interpolation of nonlinear eigenvalue problems \cite{effenberger2012chebyshev}, 
 this reduction in  computational cost is significant.

 If we consider a particular case of the above setting $k =1$, the problem 
 corresponds to approximating the roots of a scalar polynomial. 
In a recent paper Aurentz, Mach, Vandebril, and Watkins \cite{AuMaVaWa15}
 have shown that it is possible to exploit the unitary-plus-rank-one structure of the
 companion matrix to devise a backward stable $\mathcal O(d^2)$ algorithm, which is much cheaper than
 the required $\mathcal O(d^3)$  when running an iteration that does not
 exploit the structure. 
 On the other hand the case $d = 1$ is just a generalized eigenvalue problem. This can be solved 
 directly by the QZ iteration. The cost of this process equals $\mathcal
 O(k^3)$.  Based on these considerations, we believe that an asymptotic complexity
 $\mathcal O(d^2 k^3)$ is the best one can hope for, for solving this problem by means
 of a QR based approach.  In this work we will introduce an algorithm achieving this
 complexity.  Even though a fundamentally different approach from the one presented here
 might lead to a lesser complexity of, e.g., $\mathcal{O}(\max(d,k)dk^2)$, we think that a
 QZ based strategy cannot easily achieve a better result.

 There are only few other algorithms we are aware of that solve the matrix polynomial
 eigenvalue problem. Bini and Noferini \cite{BiNo13} present two versions of the Ehrlich-Aberth
 method; one is applied directly on the matrix polynomial leading to a complexity of
 $\mathcal{O}(d^2k^4)$ and another approach applied on the linearization leads to an
 $\mathcal{O}(d^3k^3)$ method.  Delvaux, Frederix, and Van Barel \cite{DeFreVB10} proposed
 a fast method to store the unitary plus low rank matrix based on the Givens-weight representation.
Cameron and Steckley \cite{Ca16} propose a technique based on Laguerre's
 iteration which is of the order $\mathcal{O}(d k^4 + d^2 k^3)$. In his PhD thesis \cite{Ro15b} Robol
 proposes a
 fast reduction of the block companion matrix to upper Hessenberg form, taking the
 quasiseparable structure into account; this could be used as a preprocessing step for
 other solvers. Eidelman, Gohberg, and Haimovici \cite[Theorems 29.1 and 29.4]{EiGoHa13} present in their monograph an algorithm for computing the
 eigenvalues of a unitary-plus-rank-$k$ matrix of total cost $\mathcal{O}(d^2 k^5)$.
 Fundamentally different techniques, via contour integration, are discussed by Van Barel
 in \cite{VB16}.

We will compute the eigenvalues of the original matrix polynomial by solving the
generalized block companion pencil\footnote{We will use both notations $(S,T)$ and
$S-\lambda T$ when referring to the pencil.}  $(S,T)$ in an efficient way.  The most important part
in the entire algorithm is the factorization of the block companion matrix $S$ and the
upper triangular matrix $T$ into the product of upper triangular or Hessenberg
unitary-plus-rank-one matrices. Relying on the results of Aurentz, et al.\ \cite{AuMaVaWa15} we can store each of these factors efficiently with only $\mathcal{O}(n)$ parameters.
Relying on the factorization we can rephrase the remaining eigenvalue problem into a
  product eigenvalue problem
\cite{benner2002perturbation,bojanczyk1992periodic,b333}. A product eigenvalue problem can
be solved by iterating on the formal product of the factors until they all converge to
upper triangular form. This procedure has been proven to be backward stable by Benner,
Mehrmann, and Xu \cite{benner2002perturbation} and the eigenvalues are formed by the product of 
the diagonal elements. The efficient storage and a core chasing implementation of the product eigenvalue
problem lead to an $\mathcal O(d^2 k^3)$ method for computing the Schur form. The
eigenvectors can be computed as well in $\mathcal O(d^2 k^3)$. Moreover, we will prove that the
algorithm is backward stable,  if suitable scaling is applied initially.

The paper is organized as follows. In Section~\ref{sec:factoring} we present three
different factorizations: we propose two ways of factoring the block companion matrix
$S$ and a way to factor the upper triangular matrix $T$. One factorization of $S$ uses
elementary Gaussian transformations, the other Frobenius companion matrices.
In Section~\ref{sec:storage} we
show how to store each of the unitary-plus-rank-one factors efficiently by
$\mathcal{O}(n)$ parameters. As the block companion matrix is not yet in Hessenberg form,
we need to preprocess the pencil $(S,T)$ to reduce it to Hessenberg-triangular form.  
We present an algorithm with complexity  $\mathcal{O}(d^2k^3)$ that achieves this in Section~\ref{sec:reduction}. 
The product
eigenvalue problem is solved in Section~\ref{sec:pep}. Deflation of infinite and zero
eigenvalues is handled in Section~\ref{sec:deflation}. Computing the eigenvectors fast is
discussed in Section~\ref{sec:eigenvectors}.
Finally we examine the backward
stability in Section~\ref{sec:backward} and provide numerical
experiments in Section~\ref{sec:numexp} validating both the stability and computational complexity of the
proposed methods.

\section{Factoring matrix polynomials} 
\label{sec:factoring}
The essential ingredient of this article is the factorization of the
companion pencil associated with a matrix polynomial. 
The factorization of the pencil coefficients into  structured matrices, providing cheap and
efficient storage allows the design of a fast structured product eigenvalue solver. 

\subsection{Matrix polynomials and pencils}
\label{sec:fmp}

Consider a degree $d$ matrix polynomial
\begin{equation}
\label{poly:nm}
\lambda^d P_d + \lambda^{d-1} P_{d-1} +
\cdots + \lambda P_1 + P_0 \in \CC^{k\times k}[\lambda],
\end{equation}
whose eigenvalues one would like to compute as eigenvalues of a pencil $S-\lambda
T\in\CC^{n\times n}[\lambda],$ with $n=dk$.
The classical linearization equals \eqref{eq:comp_pencil}. Following, however, the results of  Mackey, Mackey, Mehl, and Mehrmann
\cite{q573} and of Aurentz,
Mach, Vandebril, and Watkins \cite{AuMaVaWa15b} we know that all pencils
\begin{equation}
\label{eq:gpencil}
  S=
\left[
  \begin{array}{cccccc}
    &   & &  & -M_0 \\
     I_k &  &  & &   -M_1
\\
       & I_k &     & &  
\\
     &   &      \ddots & & \vdots\\
     &   &     & I_k & -M_{d-1}
  \end{array}
\right]
\;\;\mbox{ and }\;\;
 T=
\left[
  \begin{array}{cccccc}
     I_k &  &  & &   N_1
\\
       & I_k &     & &  
\\
     &   &      \ddots & & \vdots \\
     &   &     & I_k & N_{d-1} \\
&&&& N_{d}
  \end{array}
\right],
\end{equation}
with $I_k$ the $k\times k$ identity, $M_0=P_0$, $N_d=P_d$, and $M_i+N_i=P_i$, for all $2 \leq i \leq d-1$
have identical eigenvalues. 
In this article we discuss the more general setting \eqref{eq:gpencil}, since the
algorithm can deal with this. 
An optimal distribution of the matrix coefficients over $S$ and $T$ in terms of accuracy is, however, beyond
the scope of this manuscript, though a reasonable distribution should satisfy
$$\sqrt{\| [M_0,\ldots,M_{d-1}]\|^2+\|[N_1,\ldots,N_d]\|^2}
\approx\sqrt{\sum_{i=0}^d \| P_i \|^2 }.$$

In the next sections we will factor the matrices $S$ and $T$. Before being able to do so,
we need to preprocess the matrix polynomial. The leading coefficient $P_d$ needs to be
brought to upper triangular form and the constant term $P_0$ to upper or lower triangular
form.  This is shown pictorially in Figure~\ref{fig:nm} or Figure~\ref{fig:struc2}.
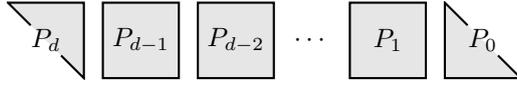
\begin{figure}[h!tb]
  \centering
  \begin{tikzpicture}[baseline={(current bounding box.center)},scale=.5,y=1cm]
    \fill[fill=gray!20!white] (-.5,0) -- (-.5,2) -- (-2.5,2) -- cycle; 
    \draw[thick] (-1.1,.6) -- (-.5,0) -- (-.5,2) -- (-2.5,2) -- (-1.9,1.4);

    \node at (-1.5,1) {$P_d$};

    \draw[thick,fill=gray!20!white] (0,0) rectangle (2,2);
    \node at (1,1) {$P_{d-1}$};
    \draw[thick,fill=gray!20!white] (2.5,0) rectangle (4.5,2);
    \node at (3.5,1) {$P_{d-2}$};
    \node at (5.5,1) {$\cdots$};
    \draw[thick,fill=gray!20!white] (6.5,0) rectangle (8.5,2);
    \node at (7.5,1) {$P_{1}$};

    \node[fill=white] at (10,1) {$\phantom{M}$};
    \fill[fill=gray!20!white] (9,0) -- (11,0) -- (9,2) -- cycle; 
    \draw[thick] (9.6,1.4)--(9,2)--(9,0) -- (11,0) -- (10.4,.6); 
    \node at (10,1) {$P_0$};
    \end{tikzpicture}

  \caption{Structure of the matrices: $P_d$ upper and $P_0$ lower triangular}
\label{fig:nm}
\end{figure}

\begin{figure}[h!tb]
  \centering
  \begin{tikzpicture}[baseline={(current bounding box.center)},scale=.5,y=1cm]
    \fill[fill=gray!20!white] (-.5,0) -- (-.5,2) -- (-2.5,2) -- cycle; 
    \draw[thick] (-1.1,.6) -- (-.5,0) -- (-.5,2) -- (-2.5,2) -- (-1.9,1.4);

    \node at (-1.5,1) {$P_d$};

    \draw[thick,fill=gray!20!white] (0,0) rectangle (2,2);
    \node at (1,1) {$P_{d-1}$};
    \draw[thick,fill=gray!20!white] (2.5,0) rectangle (4.5,2);
    \node at (3.5,1) {$P_{d-2}$};
    \node at (5.5,1) {$\cdots$};
    \draw[thick,fill=gray!20!white] (6.5,0) rectangle (8.5,2);
    \node at (7.5,1) {$P_{1}$};


    \fill[fill=gray!20!white] (11,0) -- (11,2) -- (9,2) -- cycle; 
    \draw[thick] (10.4,.6) -- (11,0) -- (11,2) -- (9,2) -- (9.6,1.4);

    \node at (10,1) {$P_0$};

    \end{tikzpicture}

\caption{Structure of the matrices: $P_d$ and $P_0$ upper triangular}
\label{fig:struc2}
\end{figure}

These structures can be achieved stably by computing the generalized Schur
decomposition\footnote{To compute the Schur decomposition of $(P_d,P_0^{-*})$ stably one solves
  the product eigenvalue problem $P_d P_0^*$ implicitly, without requiring inverses, nor
  multiplications between $P_0$ and $P_d$.} of
either  $(P_d,P_0^{-*})$ or $(P_d,P_0)$.  
Suppose $U$ and $V$ are the two unitary matrices such
that $U^*P_d V$ is upper triangular and $U^* P_0 V$ is triangular.

Performing an equivalence transformation on \eqref{poly:nm} with $U$ and $V$ 
provides the desired factorization.
In the remainder of this section we will therefore assume without loss of generality 
that all $P_i$'s are overwritten by $U^* P_i V$, such that the  leading coefficient $P_d=N_d$
is upper triangular and the  constant term $P_0=M_0$ is triangular.

\subsection{Frobenius factorization}
\label{sec:mmp}
The block companion matrix $S$ from \eqref{eq:gpencil} having its upper right block $M_0$
in lower triangular form (see Figure~\ref{fig:nm})
can be factored into the product of scalar companion matrices.

\begin{theorem}[Frobenius factorization]
\label{theo:fact}
The block companion matrix $S$ in \eqref{eq:gpencil}, 
with  blocks $M_i\in\CC^{k\times k}$ and $M_0$ lower triangular, 
can be factored as 
$S=S_1S_2 \cdots S_k$, where each $S_i\in\CC^{n\times n}$, $n=dk$, is
the Frobenius companion matrix linked to a scalar monic polynomial of degree $n$.
\end{theorem}

\begin{proof}
  The proof is constructive and proceeds recursively. We show how to factor a companion
  matrix $S$ 
into the product of two matrices $S=S_1 \widetilde{S}$, with
  $S_1$ a companion matrix of a scalar polynomial of degree $n=dk$ and $\widetilde{S}$ 
  the matrix on which we will apply the recursion.

  Consider $J$ the $n\times n$ nilpotent downshift matrix, i.e., it has ones on
  the subdiagonal and zeros elsewhere. Name
  $M=[M_0^T,M_1^T,\ldots,M_{d-1}^T]^T$ the skinny matrix
  of size $n\times k$ having all $M_i$ stacked vertically.
The elements of the blocks $M_i$ ($0\leq i\leq d-1$) are indexed according to their
position in $M$, thus $M_0$ and $M_{d-1}$ look like
\begin{eqnarray*}
\left[
  \begin{array}{cccc}
    m_{11}  \\
    m_{21} & m_{22} & \\
    \vdots & & \ddots \\
    m_{k1} & m_{k2} & \cdots & m_{kk} \\
  \end{array}
\right]
\;\;\mbox{and}\;\;
\left[
  \begin{array}{cccc}
    m_{(d-1)k+1,1} & m_{(d-1)k+1,2} & \cdots & m_{(d-1)k+1,k}  \\
    m_{(d-1)k+2,1} & m_{(d-1)k+2,2} &        & m_{(d-1)k+2,k} \\
    \vdots         &                & \ddots  & \vdots \\
    m_{dk,1} & m_{dk,2} & \cdots & m_{dk,k} \\
  \end{array}
\right].
\end{eqnarray*}

 The block companion matrix $S=J^k -    
M\left[\;0,\; \ldots,\;0,\;I_k \right]
$ admits a factorization $S_1 \widetilde{S}$, where 
$    S_1   =   
      J  - 
    \left[
        m_{11},
        m_{21},
        \ldots,
        m_{dk,1}
      \right]^T\, e_{n}^T$,
 and 
    $\widetilde{S}   =   
      J^{k-1} - \widetilde{M} \left[\;0,\; \ldots,\;0,\;I_{k-1} \right]\;,
$
with 
$e_{n}$  the $n$-th
 standard basis vector and
$\widetilde{M}$  of size $n\times (k-1)$ consisting of the last $k-1$ columns of
$M$ moved up one row. The top $k$ and bottom $k$ rows of $\widetilde{M}$ are thus of the forms 
\begin{eqnarray*}
\left[
  \begin{array}{cccc}
    m_{22}  \\
    m_{32} & m_{33} & \\
    \vdots & & \ddots \\
    m_{k,2} & m_{k,3} & \cdots & m_{k,k} \\
    m_{k+1,2} & m_{k+1,3} & \cdots & m_{k+1,k} \\
  \end{array}
\right]
\;\;\mbox{and}\;\;
\left[
  \begin{array}{cccc}
    m_{(d-1)k+2,2} & m_{(d-1)k+2,3} & \cdots & m_{(d-1)k+2,k}  \\
    m_{(d-1)k+3,2} & m_{(d-1)k+3,3} &        & m_{(d-1)k+3,k} \\
    \vdots         &                & \ddots  & \vdots \\
    m_{dk,2} & m_{dk,3} & \cdots & m_{dk,k} \\
    0 & 0 & \cdots & 0
  \end{array}
\right].
\end{eqnarray*}

Continuing to factor $\widetilde{S}$ in a similar way results in the desired factorization $S=S_1\cdots
S_k$, where each $S_i$ is the Frobenius companion linearization of a degree $n=dk$ scalar polynomial.
\end{proof}
\begin{remark}
Computing the Frobenius factorization is
  cheap and stable
as no arithmetic operations are involved.
\end{remark}

\begin{example}
Consider a degree $3$ matrix polynomial $\x^3 I_3 + \x^2 M_2 + \x M_1 + M_0$ with $3\times 3$ blocks
\begin{equation*}
M_0=-\left[
\begin{array}{ccc}
1 &  &  \\
2 & 10 &  \\
3 & 11 & 18 \\
\end{array}
\right],\;\;
M_1=-\left[
\begin{array}{ccc}
4 & 12 & 19 \\
5 & 13 & 20 \\
6 & 14 & 21 \\
\end{array}
\right],
\;\;\mbox{and}\;\;
M_2=-\left[
\begin{array}{ccc}
7 & 15 & 22 \\
8 & 16 & 23 \\
9 & 17 & 24 \\
\end{array}
\right].
\end{equation*}
  The associated companion matrix $S$ equals $S_1S_2S_3$, with
  \begin{equation*}
\small
    S_1 = 
\left[
  \begin{array}{cccccc}
      &  &  & &&1 \\
    1 &  &  & &&2 \\
      & \ddots&&& &\vdots  \\
      && 1 &&&   7 \\
      &&& 1 &&   8 \\
      & & && 1   & 9
  \end{array}
\right],\;
S_2 = 
\left[
  \begin{array}{ccccccc}
      &        &  && &  10 \\
    1 &        &  && &  11 \\
      & \ddots &  && &  \vdots  \\
      &        &  1 && & 16  \\
      &        &  & 1& &  17 \\
      &        &  & &1&  0
  \end{array}
\right],\;
    S_3 = 
\left[
  \begin{array}{cccccc}
      &        &  & & & 18 \\
    1 &        &  & & & 19 \\
      & \ddots &  & & & \vdots  \\
      &        & 1& & & 24 \\
      &        &  &1& & 0 \\
      &        &  & &1& 0 
  \end{array}
\right].
  \end{equation*}

\end{example}

In Section~\ref{sec:abc} we will provide an alternative factorization of the block
companion matrix $S$, which, as we will show in Section~\ref{sec:backward} leads to a 
 slightly better bound on the backward error.

\subsection{Gaussian factorization}

Both $S$ and $T$ are of unitary-plus-low-rank form, which is essential for developing a
fast algorithm. For proving backward stability in Section~\ref{sec:backward}, we need,
however, more; we require that all the factors in the factorizations of $S$ and $T$ have the low rank part
concentrated in a single column, the \emph{spike}. For instance a companion is also a
unitary-plus-spike matrix.  For factoring $T$ from \eqref{eq:gpencil}, we will use
Gaussian transformations providing us a factorization into \emph{identity-plus-spike
  matrices}, that is, matrices that can be written as the sum of the identity and a
rank-one matrix having a
single nonzero column.

\begin{theorem}[Gaussian factorization]
\label{theo:gc}  

The matrix $T$ from \eqref{eq:gpencil}, with matrix blocks $N_i\in\mathbb{C}^{k\times k}$
and 
$N_d$ upper triangular %
can be factored
as $T=T_1T_2 \cdots T_k$, where each $T_i\in\CC^{n\times n}$, with $n=dk$, is
upper triangular and of identity-plus-spike form.

\end{theorem}
\begin{proof}
The proof is again of recursive and constructive nature and involves elementary matrix
operations. We will factor 
$T$ 
as $\tilde{T} T_k$,
where $T_k$ is of the desired upper triangular and identity-plus-spike form, 
the recursion will be applied on the matrix $\tilde{T}$.
Write 
\begin{equation*}
T  =I_{n}+ 
[
    N_1^T,
    N_2^T, 
    \ldots, 
    N_{d-1}^T, 
    \tilde{N}_d^T
]^T
[0, \ldots, 0, I_k]
= I_n+\tilde{N}\;[0, \ldots, 0, I_k],
\end{equation*}
with $\tilde{N}\in\CC^{n\times k}$ and $\tilde{N}_d=N_d-I_k$. 
Again we index the individual blocks $N_i$ ($1\leq i\leq d-1$) and $\tilde{N}_d$  according to $N$.  It is
easily verified that $T=\widetilde{T}T_k$, with
\begin{equation*}
T_k = I_n+
\left[
    n_{11}, 
    n_{21}, 
    \ldots, \\
    \tilde{n}_{(d-1)k+1,1}, \\
    0,\ldots,0
\right]
 e_{(d-1)k+1}^T
= (\tilde{N} e_1)  e_{(d-1)k+1}^T,
\end{equation*}
 and 
\begin{equation*}
\widetilde{T}=I_n+
\left[
  \begin{array}{cccc}
    n_{12} & n_{13} & \cdots & n_{1k} \\
    n_{22} & n_{23} & \cdots & n_{2k} \\
    \vdots & \vdots &  & \vdots \\
    n_{(d-1)k,2} & n_{(d-1)k,3} & \cdots & n_{(d-1)k,k} \\
    \tilde{n}_{(d-1)k+1,2} & \tilde{n}_{(d-1)k+1,3} & \cdots & \tilde{n}_{(d-1)k+1,k} \\
    \tilde{n}_{(d-1)k+2,2} & \tilde{n}_{(d-1)k+2,3} & \cdots & \tilde{n}_{(d-1)k+2,k} \\
             & \tilde{n}_{(d-1)k+3,3} & \cdots & \tilde{n}_{(d-1)k+3,k} \\
             &     & \ddots & \vdots \\
             &     &        & \tilde{n}_{dk,k} \\
  \end{array}
\right]
[0_{(k-1) \times ((d-1)k+1)} , I_{k-1}].
\end{equation*}

Continuing to factor $\widetilde{T}$ proves the theorem.

 \end{proof}

\begin{remark}
  There are two important consequences of having $P_d=N_d$ in upper triangular form. First,
  the proof of Theorem~\ref{theo:gc} reveals that we can compute the factorization in an
  exact way, as no computations are involved. Second, all factors in the
  factorization of $T$ are already in upper triangular form, which will come in handy in
  Section~\ref{sec:reduction} when
  transforming the pencil $(S,T)$ to Hessenberg-triangular form.
\end{remark}

\begin{example}
\label{ex:gf}
  Consider a $9\times 9$ upper triangular matrix $T$ as in \eqref{eq:gpencil}, with
  $3\times 3$ blocks
\begin{equation*}
  N_2=
\left[
\begin{array}{ccc}
1 & 8 & 16  \\
2 & 9 & 17 \\
3 & 10 & 18
\end{array}
\right]
, \quad N_3=
\left[
\begin{array}{ccc}
4 & 11 & 19 \\
5 & 12 & 20 \\
6 & 13 & 21 \\
\end{array}
\right]
, \;\mbox{and}\;\quad N_4 =
\left[
\begin{array}{ccc}
7 & 14 & 22 \\
  & 15 & 23 \\
  &    & 24
\end{array}
\right]
. 
\end{equation*}
Then we have that $T=T_1 T_2 T_3$ with
  \begin{equation*}
\small
    T_1 = 
\left[
  \begin{array}{cccccc}
    1  &        &      && & 16 \\
       & \ddots &      && & \vdots \\
       &        & 1&& & 21  \\
       &        & & 1 & & 22 \\
       &   &      & &1& 23 \\
       &   &      & & & 24
  \end{array}
\right],\; 
T_2 =
\left[
  \begin{array}{cccccc}
    1  &        &      && 8 &  \\
       & \ddots &      && \vdots &  \\
       &        & 1&& 13&   \\
       &        & & 1 & 14&  \\
       &   &      & &15&  \\
       &   &      & & & 1
  \end{array}
\right],\; 
    T_3 = 
\left[
  \begin{array}{cccccc}
    1  &   &      & 1 \\
       & \ddots &      & \vdots \\
       &   &1& 6 \\
       &   &      & 7& \\
       &   &      &  & 1 \\
       & & & & & 1
  \end{array}
\right].
  \end{equation*}

\end{example}



\subsection{Gaussian factorization of the block companion matrix}
\label{sec:abc}

In this section we provide a second option for factoring the block companion matrix $S$ (see \eqref{eq:gpencil}).
Whereas Theorem~\ref{theo:fact} factors $S$ into regular
Frobenius companion matrices, the  factorization proposed here relies only on
Theorem~\ref{theo:gc} for factoring upper triangular matrices. To be able to do so
the block $M_0$ needs to be in upper triangular form (see Figure~\ref{fig:struc2}),
in contrast to the lower triangular form required by Theorem~\ref{theo:fact}.

Instead of factoring $S$ directly we compute its QR factorization first. Let ($n=dk$)
\begin{eqnarray}
\label{matrixQ}
  \CS{Q} = \left[\begin{array}{ccccc}
      0 & 0  & \cdots        & 0  &  1     \\
      1 &   &        &   &  0       \\
      & 1 &        &   &  0        \\
      &   & \ddots &  & \vdots   \\
      & & & 1 & 0 
    \end{array}\right]\in\CC^{n \times n},
\end{eqnarray}
leading to 
\begin{eqnarray*}
  S=\CS{Q}^k R = \left[\begin{array}{ccccc}
      0 & 0  & \cdots        & 0  &  I_k     \\
      I_k &   &        &   &  0       \\
      & I_k &        &   &  0        \\
      &   & \ddots &  & \vdots   \\
      & & & I_k & 0 
    \end{array}\right]
\left[\begin{array}{ccccc}
      I_k &   &         &   &  -M_1     \\
       & I_k  &        &   &    -M_2  \\
      &  & \ddots       &   &  \vdots      \\
      &   &  & I_k & -M_{d-1}   \\
      & & &  & -M_0 
    \end{array}\right].
\end{eqnarray*} 
The upper triangular matrix $R$ can be factored by Theorem~\ref{theo:gc} as $R=R_1\cdots
R_k$, providing a factorization of $S$ of the form $S = \mathcal Q^k R_1 \cdots R_k$.

\section{Structured storage}
\label{sec:storage}

To develop an efficient product eigenvalue algorithm we will exploit the structure 
 of the factors.  In addition to being sparse, 
 all factors are also of unitary-plus-rank-one form, and we can use
an $\mathcal{O}(n)$ storage scheme, where $n=dk$. We recall how to efficiently store
these matrices, since the upper triangular factors differ slightly from what is
presented by Aurentz et al.\ \cite{AuMaVaWa15}.  

To factor these matrices efficiently we use \textit{core transformations} $Q_i$, which are
identity matrices except for a unitary two by two block $Q_i(i:i+1,i:i+1)$. Note that the
subscript $i$ in core transformations will always refer to the \textit{active part} $(i:i+1,i:i+1)$.
 For instance
rotations or reflectors operating on two consecutive rows are core transformations.

Let us first focus on the factorization $S=S_1\cdots S_k$ from Section~\ref{sec:mmp}.
All matrices $S_i$ with $1\leq i\leq k$ are of
unitary-plus-rank-one and of Hessenberg form and we replace them by their QR factorization 
$S_i=\CS{Q} R_i$. This factors the companion matrices into a unitary $\CS{Q}$ and upper
triangular identity-plus-spike matrices $R_i$. Indeed, the matrix $\CS{Q}$ is, initially, identical for all $S_i$, and looks like \eqref{matrixQ},
being the product of $n-1$ core transformations:
$\CS{Q}=Q_{1}\dotsm Q_{n-1}$, with all $Q_{i}(i:i+1,i:i+1)=\left[\begin{smallmatrix} 0 &
    1\\ 1 & 0\end{smallmatrix}\right]$ counteridentities. From now on, we will use
caligraphic letters to denote a sequence of core transformations such as $\CS{Q}$.
Because the core transformations in $\mathcal Q$ are ordered such that the first one acts
on rows $1$ and $2$, the second one on rows $2$ and $3$, the third one on rows $3$ and $4$,
and so forth, we will refer to it as a \emph{descending sequence} of core transformations.
When presenting core transformations pictorially as in Section~\ref{sec:reduction}
we see that $\mathcal Q$ clearly describes a descending pattern of core transformations.

As a result we have a factorization for $S$ of the form
\begin{equation}
  \begin{array}{rcll}
  S & = & S_1 S_2 \cdots S_k \; = \; \CS{Q} R_1\, \CS{Q} R_2\, \cdots\, \CS{Q} R_k &\quad 
(\mbox{Frobenius factorization}).
\end{array}
\label{f:1}
\end{equation}
Alternatively, considering the factorization from Section~\ref{sec:abc}, we would get 
\begin{equation}
  \begin{array}{rcll}
  S & = & \CS{Q}^k R_1 \cdots R_k &\quad 
(\mbox{Gaussian factorization}).
\end{array}
\label{f:2}
\end{equation}
It is important to notice that even though \eqref{f:1} and \eqref{f:2}
use the same symbols $R_i$, they are different. 
However, as both factorizations are never
used simultaneously, we will reuse these symbols throughout the text.

Since $\CS{Q}$ can be factored into $n-1$ core transformations
$\CS Q^k$ from \eqref{matrixQ} admits a factorization into $k$
\emph{descending sequences} of core transformations too. It remains to
efficiently store the upper triangular identity-plus-spike matrices
$R_i$:
\begin{equation}
\label{eq:ri}
R_i=\left[\begin{array}{cccccccc}
      1 &   &        &   & \times      \\
      & 1 &        &   &  \times       \\
      &   & \ddots &   & \vdots   \\
      &   &        & 1 & \times   \\
      &   &        &   & \times   \\
      &   &        &   &       & 1 \\
      & & & & && \ddots\\
      & & & & && &1 
    \end{array}\right].
\end{equation}
In the Frobenius case \eqref{f:1} the spike is always located in the last column and in
Gaussian case \eqref{f:2}
the spike is found in column $n+1-i$.

For simplicity we drop the subscript $i$ and we deal with all cases at once: 
Let $R$ have the form \eqref{eq:ri},  
with  the spike in column $\ell$.   Then $R=Y+(x - e_\ell )y^T$ with 
$Y=I_n$ unitary, $x\in\mathbb{C}^{n}$ the vector containing the spike,  and $y= e_{\ell} \in \mathbb{C}^{n}$. 
We begin by
embedding $R$ in a larger matrix  $\underline{R}$
having one extra zero row and an extra column with only a single element different from zero:
$$ \underline{R}=\left[\begin{array}{cccccccc|c}
      1 &   &        &   & \times &&&& 0    \\
      & 1 &        &   &  \times &&&& 0     \\
      &   & \ddots &   & \vdots &&&& \vdots \\
      &   &        & 1 & \times &&&& 0  \\
      &   &        &   & \times &&&& 1  \\
      &   &        &   &       & 1 &&& 0 \\
      & & & & & & \ddots && \vdots\\
      & & & & & & &1& 0 \\ \hline
      0& 0 & \cdots & 0 &0 & 0 &\cdots &0 & 0
    \end{array}\right].
$$
Let $P$ denote the $(n+1)\times n$ 
matrix obtained by adding a zero row to the bottom of the 
$n \times n$ identity matrix.  Then $R = P^{T}\underline{R}P$.   In fact our object of interest is $R$, but for 
the purposes of efficient storage we need the extra room provided by the larger matrix
$\underline{R}$. 
The $1$ in the last column ensures that the matrix $\underline{R}$ remains of
unitary-plus-spike form $\underline{R}=\underline{Y} + \underline{x}\underline{y}^T$, 
where 
$\underline{y}=e_\ell\in\mathbb{C}^{n+1}$, 
and $\underline{x}\in\mathbb{C}^{n+1}$ is just 
$x$ with a $-1$ adjoined at the bottom. Thus $P^T \underline{y}= y$, $P^T \underline{x}=x$,
\begin{equation}
\underline{Y}=\left[\begin{array}{cccccccc|c}
      1 &   &        &   &  &&&&    \\
      & 1 &        &   &  &&&&      \\
      &   & \ddots &   &  &&&&  \\
      &   &        & 1 &  &&&&   \\
      &   &        &   & 0 &&&& 1  \\
      &   &        &   &       & 1 &&&  \\
      & & & & & & \ddots && \\
      & & & & & & &1&  \\ \hline
      &  &  &  & 1 &  & & & 0
    \end{array}\right] \;\;\mbox{ and }\;\;
\underline{x}=
\left[
  \begin{array}{c}
    \times \\ \times \\ \times \\
    \vdots \\
    \times \\
    0 \\
    \vdots\\
    0 \\ \hline
     -1\\
  \end{array}
\right].
\label{eq:x:nw}
\end{equation}

Let  $C_{1},\dotsc,C_{n}$  be core transformations such that 
$C_{1} \dotsm C_{n} \underline{x} = \alpha e_{1}$, with $\absval{\alpha} = \norm{\underline{x}}_{2}$. 
Because of the nature of $\underline{x}$,  each of the core transformations $C_{\ell+1},\dotsc,C_{n}$ 
has the simple active part $F = 
\left[\begin{smallmatrix} 0 &
    1\\ 1 & 0\end{smallmatrix}\right]
$.  If we 
use the symbol $F_{j}$ to denote a core transformation with active part $F$, then $C_{j} = F_{j}$ for
$j=\ell+1$, \ldots, $n$.   
Let $\CS{C} = C_{1} \dotsm C_{n}$. 
Then we can write 
$$\underline{R} = \CS{C}^*(\CS{C}\underline{Y} + e_{1}\underline{y}^{T}),$$
where we have absorbed the factor\footnote{In the remainder of the text we will always
  assume $\|\underline{x}\|=1$ and as a consequence that $\alpha$ is absorbed into $\underline{y}$.}
$\alpha$ into $y$.  Let $\CS{B} = \CS{C}\underline{Y} = C_{1} \dotsm C_{n} \underline{Y}$.
It is easy to check that $C_{\ell+1} \dotsm C_{n} \underline{Y} = F_{\ell+1} \dotsm F_{n} \underline{Y} = F_{\ell} F_{\ell+1} \dotsm F_{n}$.
Thus $\CS{B} = C_{1} \dotsm C_{\ell} F_{\ell} \dotsm F_{n}$, so we can write $\CS{B} = B_{1} \dotsm B_{n}$, where 
$B_{\ell} = C_{\ell}F_{\ell}$,  and $B_{j} = C_{j}$ for $j \neq \ell$.  
As a result we get
\begin{eqnarray}
R = P^{T}C_{n}^*\dotsm C_{1}^*( B_{1}\dotsm B_{n} + e_{1}\underline{y}^{T})P = P^T \CS{C}^* (\CS{B}
+ e_1 \underline{y}^T ) P.
\label{eq:factorize:upper:triangular}
\end{eqnarray} 
The symbols $P$ and $P^{T}$ are there just to make the dimensions come out right.  They add nothing to 
the storage or computational cost, and we often forget that they are there.  

\begin{remark}
It was shown in \cite{AuMaVaWa15}
 that by adding an extra row and column to the
  matrix $R$, preserving the unitary-plus-spike and the upper triangular structure,
  that all the information about the rank-one part is encoded in the unitary part.   Thus we will
  not need to store the rank-one-part explicitly. We will consider this in more detail in
  Section~\ref{sec:backward} when discussing the backward error.
\end{remark}

The overall computational complexity for efficiently storing the factored form of the
pencil equals $\mathcal{O}(dk^3)$ subdivided in the following parts.

The preprocessing step to bring the constant and leading coefficient matrix to suitable
triangular form requires solving a product eigenvalue problem or computing a Schur
decomposition which costs $\mathcal{O}(k^3)$ operations. Also the other matrices require
updating: $2(d-1)$ matrix-matrix products, assumed to take $\mathcal{O}(k^3)$ each are
executed.  Factoring  $T$ and $S$ is for free since no arithmetic operations
are involved.  Overall, computing the factored form requires thus $\mathcal{O}(dk^3)$
operations.


The cost of computing the efficient storage of a single identity-plus-spike matrix is
  essentially the one of computing $\CS{C}\underline{x}=\alpha e_1$. This requires computing $n=dk$ core
  transformations for 
$2k$ matrices. In total this sums up to $\mathcal{O}(dk^2)$.

\section{Operating with core transformations}
\label{sec:oper}

The proposed factorizations are entirely based on core transformations; we need three
basic operations to deal with them: a fusion, a turnover, and a pass-through operation.
To understand the flow of the algorithm better we will explain it with pictures.  
A core transformation is therefore pictorially denoted as $\begin{smallmatrix} \Rc\\[0.7ex] \rc
\end{smallmatrix}$, where the arrows pinpoint the rows affected by the transformation.

Two core transformations $F_i$ and $G_i$ undergo a \emph{fusion} when they operate
on identical rows and can be replaced by a single core tranformation $H_i=F_i G_i$.
Pictorially this is shown on the left of \eqref{fig:op:1}.
Given a product 
$F_{i} G_{i+1} H_{i}$ of three core transformations, then one can always  
refactor the product  as 
${F}_{i+1} {G}_{i}  {H}_{i+1} $. This operations is called a \emph{turnover}.\footnote{This
  is proved easily by considering the QR or QL factorization of a $3\times 3$ unitary
  matrix. We refer to Aurentz et al.\ \cite{AuMaVaWa15} for more details and to the \texttt{eiscor} package https://github.com/eiscor/eiscor for a reliable implementation.}  
This is shown pictorially on the right of
\eqref{fig:op:1}. 
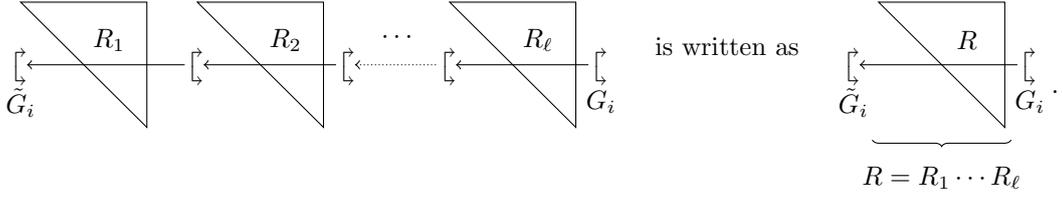
  \begin{eqnarray}
  \begin{matrix}\Rc & \Rc\\\rc& \rc\end{matrix}\ 
= \ \begin{matrix}\Rc\\ \rc\end{matrix}\;,
\quad\quad\quad\quad
  \begin{matrix}
    \Rc & &\Rc\\
    \rc &\Rc &\rc\\
    &\rc &\\
  \end{matrix} \ = 
  \begin{matrix}
    &\Rc&\\
    \Rc &\rc &\Rc\\
    \rc& &\rc\\
  \end{matrix}\;.
\label{fig:op:1}
\end{eqnarray} 

The final operation involves core transformations and upper triangular matrices. A core
transformation can be moved from one side of an upper triangular matrix to the other
side: $R G_i = \tilde{G}_i \tilde{R} $, named a \emph{pass-through} operation.

When describing an algorithm with core transformations, typically one core transformation
is more important than the others and it is desired to move this transformation around.
Pictorally we represent the movement of a core transformation by an arrow. For instance
\begin{equation}
\label{eq:sp}
  \begin{tikzpicture}[scale=1.66,y=-1cm]
    \tikzrotation{-1.2}{0.2}
    \tikzrotation{-1.0}{0.0}
    \tikzrotation{-0.8}{0.2}
    \tikzrotation{-0.6}{0.0}
    \turnoverrl{-0.6}{0.0}{-1.2}{0.2}

    \node[above right] at (-.5,0.4) {,};

     \tikzrotation{.4}{0.0}
     \tikzrotation{.6}{0.2}
     \tikzrotation{.8}{0.0}
     \tikzrotation{1.0}{0.2}
     \turnoverrl{1.0}{0.2}{.4}{0.0}
     \node [above] at (1.2,0.4) {,};
    \tikzrotation[black][][]{1.8}{0.1}
    \tikzrotation[black][][]{2.2}{0.1}
    \path[->,out=180,in=0] (2.1,0.2) edge (1.85,0.2);     
   \node [above] at (2.4,0.4) {,};
  \end{tikzpicture}
\end{equation}
demonstrates pictorially the movement of a core transformation from the right to the left as the result of executing
a turnover (left and middle picture of \eqref{eq:sp}) and
a fusion (right picture of \eqref{eq:sp}), where the right rotation is fused with the one on the left.
We need pass-through operations in both directions, pictorially shown as
\begin{center}
  \begin{tikzpicture}[baseline={(current bounding box.center)},scale=1.5,y=-1cm]
    \uppertriangular{0.0}{0.0}{}; 
    \tikzrotation{1.4}{.4}{}
    \tikzrotation{-.2}{.4}{}
    \shiftthroughrl{1.4}{-.2}{0.4} 

    \node [above] at (2.05,0.7) {or};

    \uppertriangular{3}{0.0}{}; 
    \tikzrotation{4.4}{.4}{}
    \tikzrotation{2.8}{.4}{}
    \shiftthroughlr{2.8}{4.4}{0.4} 

    \node[above] at (4.6,0.6) {.};
  \end{tikzpicture}
\end{center}


In the description of the forthcoming algorithms we will also \textit{pass core
  transformations through the inverses of upper triangular matrices}. In
case of an invertible upper triangular matrix $R$, this does not pose any problems; numerically, however, this is unadvisable as we do not wish to invert $R$.
So instead of computing $R^{-1} G_i =\tilde{G}_i \tilde{R}^{-1}$ we compute $G^*_i R =
\tilde{R} \tilde{G}^*_i$, which can be executed numerically reliably (even when $R$ is
singular). 
 Pictorially
\begin{center}
   \begin{tikzpicture}[baseline={(current bounding box.center)},scale=1.66,y=-1cm]

    \tikzrotation[black][][$\tilde{G}_i$]{0.0}{0.6};
    \shiftthroughrl{1.4}{0.0}{0.6};
    \tikzrotation[black][][${G_i}$]{1.4}{0.6};
    \uppertriangular{0.0}{0.2}{$R^{-1}$};

    %
    \node[above] at (2.4,.7) {is computed as};    

    \tikzrotation[black][][${G}^*_i$]{3.5}{0.6};
    \shiftthroughlr{3.5}{4.9}{0.6};
    \tikzrotation[black][][$\tilde{G}_i^*$]{4.9}{0.6};
    \uppertriangular{3.5}{0.2}{$R$};

    \node[above] at (5.1,.7) {.};    

  \end{tikzpicture}
\end{center}

Not only will we pass core transformations through the inverses of upper triangular
matrices, we will also pass them through sequences of upper triangular matrices. Suppose,
e.g., that we have a product of $\ell$ upper triangular matrices $R_1\cdots R_\ell$ and we want
to pass $G_i$ from the right to the left through this sequence. Passing the core
transformation sequentially through $R_\ell$, $R_{\ell-1}$, up to $R_1$ provides us 
$
  (R_1\cdots R_\ell) G_i= \tilde{G}_i (\tilde{R}_1 \cdots \tilde{R}_\ell).
$
Or, simply writing $R=R_1\cdots R_\ell$ we have $RG_i=\tilde{G}_i\tilde{R}$, what is exactly
what we will often do to simplify the pictures and descriptions. Pictorially 

\begin{center}
   \begin{tikzpicture}[baseline={(current bounding box.center)},scale=1.66,y=-1cm]

    \tikzrotation[black][][$\tilde{G}_i$]{0.0}{.6};
    \shiftthroughrl{1.4}{0.0}{0.6};
    \uppertriangular{0.0}{0.2}{$R_1$};
    \tikzrotation{1.4}{.6};
    \shiftthroughrl{2.6}{1.4}{0.6};
    \uppertriangular{1.4}{0.2}{$R_2$};
    \tikzrotation{2.6}{.6};
    \node[above] at (3.0,.6) {$\cdots$};
    \draw[<-,densely dotted] (2.65,.7) -- (3.3,.7);
    \tikzrotation{3.4}{.6};
    \shiftthroughrl{4.6}{3.4}{0.6};
    \uppertriangular{3.4}{0.2}{$R_\ell$};
    \tikzrotation[black][][$G_i$]{4.6}{.6};

    %
     \node[above] at (5.6,.7) {is written as};    
     
     \tikzrotation[black][][$\tilde{G}_i$]{6.6}{.6};
     \shiftthroughrl{8.0}{6.6}{0.6};
     \uppertriangular{6.8}{.2}{$R$};
     \tikzrotation[black][][$G_i$]{8.0}{.6};
     
     \drawbrace{6.8}{7.8}{1.3}{$R=R_1\cdots R_\ell$}


    \node[above] at (8.2,1.0) {.};    

  \end{tikzpicture}
\end{center}

One important issue remains. We have described how to pass a core transformation from
one side to the other side of an upper triangular matrix. In practice, however, we do not
have dense upper triangular matrices, but a factored form as presented in
\eqref{eq:factorize:upper:triangular}, which we are able to update easily.
Since the rank-one part can be recovered 
from the
unitary matrices $\CS{C}$ and $\CS{B}$,  we can ignore it; 
passing a core transformation through the upper triangular matrix $R=
P^T \CS{C}^*(\CS{B} + e_{1}\underline{y}^{T}) P$ can be replaced by passing a core transformation
through the unitary part $\CS{C}^*\CS{B}$ only. So 
$RG_i=\tilde{G}_i\tilde{R}$ is computed as $
(\CS{C}^* \CS{B}) \underline{G}_i 
= \underline{\tilde{G}}_i(\tilde{\CS{C}}^* \tilde{\CS{B}})
$, where 
$\underline{G}_i=\left[\begin{smallmatrix} G_i &
    0\\ 0 & 1\end{smallmatrix}\right]
$ and 
$\underline{\tilde{G}}_i=\left[\begin{smallmatrix} \tilde{G}_i &
    0\\ 0 & 1\end{smallmatrix}\right]
$ and thus
$\underline{G}_i P = P G_i$ and
$\underline{\tilde{G}}_i P = P \tilde{G}_i$.
The pass-through operation
requires two turnovers as
pictorially shown, 
for $d=2$, $k=3$, and $n=dk=6$,
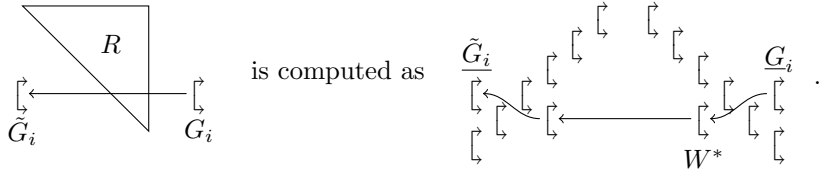
\begin{equation*}
    \begin{tikzpicture}[baseline={(current bounding box.center)},scale=1.66,y=-1cm]

      \tikzrotation[black][][$\tilde{G}_i$]{0.0}{0.6}; 
      \shiftthroughrl{1.4}{0.0}{0.6};
      \tikzrotation[black][][${G_i}$]{1.4}{0.6}; 
      \uppertriangular{0.0}{0.0}{$R$};

    %
      \node[above] at (2.5,.7) {is computed as};

    \foreach \j in {0,...,5}{
          \tikzrotation{\j/5+5.0}{\j/5}
    }      

    \foreach \j in {-1,...,4}{
          \tikzrotation{\j/5+3.8}{-\j/5+.8}
    }      

    \node at (3.6,0.6) [above] {$\underline{\tilde{G}_i}$};
     \tikzrotation[black][][]{3.6}{0.6}; 
     \tikzrotation[black][][]{4.2}{0.8}; 
     \turnoverrl{4.2}{0.8}{3.6}{0.6};

     \shiftthroughrl{5.4}{4.2}{.8};

     \tikzrotation[black][][]{5.4}{0.8}; 
     \turnoverrl{6.}{.6}{5.4}{.8};
     \node at (6.,0.6) [above] {$\underline{{G}}_i$};
     \node at (5.4,1.07) [below] {${W^*}$};
     \tikzrotation[black][][]{6.}{0.6}; 

      \node[above] at (6.3,.7) {.};s
    \end{tikzpicture}
\end{equation*}
We emphasize that $\underline{G}_i e_{n+1} = \underline{\tilde{G}}_i e_{n+1}=e_{n+1}$ and
$W^*e_1=e_1$ will always hold and these relations are used in the backward error analysis.

All these operations, i.e., a turnover, a fusion, and a pass-through, require a constant
$\mathcal{O}(1)$ number of arithmetic operations and are thus independent of the matrix
size. As a result, passing a core transformation through a sequence of compactly stored
upper triangular unitary-plus-rank-one matrices costs $\mathcal{O}(k)$, where $k$ is
the number of upper triangular matrices involved. Moreover, the turnover and the fusion
are backward stable operations \cite{AuMaVaWa15}, they introduce only errors of the order of the machine
precision on the original matrices. The stability of a pass-through operation involving
factored upper-triangular-plus-rank-one matrices will be discussed in
Section~\ref{sec:backward}.

\section{Transformation to Hessenberg-triangular form}
\label{sec:reduction}
The companion matrix $S$ has $k$ nonzero subdiagonals. To efficiently compute the
eigenvalues of the pencil $(S,T)$ via Francis's implicitly shifted QR algorithm
\cite{Fr61,Fr62} we need a unitary equivalence to transform the pencil to
Hessenberg-triangular form. 


We will illustrate the reduction procedure on the Frobenius factorization \eqref{f:1}, though it can
be applied equally well on the Gaussian factorization \eqref{f:2}.
We search for unitary matrices $U$ and $V$ such that the pencil ($U^* S V $, $U^* T V$) is
of Hessenberg-triangular form. 
We operate directly on the factorized versions of $S$ and $T$ in $U^* S V V^* T^{-1} U$
and we compute 
\begin{eqnarray*}
  & & (U^* S_1 U_1)\; (U_1^* S_2 U_2)\;  \cdots\; (U_{k-1}^* S_k V)\;  (V^* T_k^{-1} V_{k-1})\;
  (V_{k-1}^* T_{k-1}^{-1} V_{k-2})\; \cdots\; (V_1^* T_1^{-1} U) \\
& = & \tilde{S}_1 \; \tilde{S}_2\;  \cdots\; \tilde{S}_k \; \tilde{T}_k^{-1}\;\tilde{T}_{k-1}^{-1}\; \cdots\; \tilde{T}_{1}^{-1}
, 
\end{eqnarray*}
where $\tilde{S}_1$ is Hessenberg and all other factors $\tilde{S}_2,\ldots
\tilde{S}_k,\tilde{T}_1,\ldots,\tilde{T}_k$ are upper triangular. Every time we have to
pass a core transformation through an $n\times n$ upper triangular factor we
use the efficient representation and the tools described in Section~\ref{sec:oper} to pass it through the product and the
inverse factors. 

We illustrate the procedure on a running example with $k=3$ and $d=2$, so the matrices
are of size $6\times 6$ and the product is of the form 
\begin{equation*}
  \begin{tikzpicture}[baseline={(current bounding box.center)},scale=1.5,y=-1cm]
    \upperhessenberg{0.0}{0.0}{$S_1$};
    \upperhessenberg{1.4}{0.0}{$S_2$};
    \upperhessenberg{2.8}{0.0}{$S_3$};
    \uppertriangular{4.2}{0.0}{$T_3^{-1}$}; 
    \uppertriangular{5.6}{0.0}{$T_2^{-1}$}; 
    \uppertriangular{7.0}{0.0}{$T_1^{-1}$}; 
    \drawbrace{4.2}{8.0}{1.1}{$T^{-1}=T_3^{-1} T_2^{-1} T_1^{-1}$};
    \drawbrace{0.0}{3.8}{1.1}{$S=S_1S_2S_3$};
    \node at (8.2,.5) {.};
  \end{tikzpicture}
%
\end{equation*}
However, we do not work on these Hessenberg matrices, but directly on their QR
factorizations. Pictorially, 
where for simplicity of presentation
we have replaced
$T_3^{-1} T_2^{-1} T_1^{-1}$ by $T^{-1}$ we get
\begin{equation*}
  \begin{tikzpicture}[baseline={(current bounding box.center)},scale=1.5,y=-1cm]
  \foreach \j in {0.0,0.2,0.4,...,.8} {
    \tikzrotation{\j-0.6}{\j}
  }      
  \uppertriangular{-.4}{0.0}{$R_1$};

  \foreach \j in {0.0,0.2,0.4,...,.8} {
    \tikzrotation{\j+1.0}{\j}
  }      
  \uppertriangular{1.2}{0.0}{$R_2$};

  \foreach \j in {0.0,0.2,0.4,...,.8} {
    \tikzrotation{\j+2.6}{\j}
  }      
  \uppertriangular{2.8}{0.0}{$R_3$};

  \uppertriangular{4.2}{0.0}{$T^{-1}$}; 
  \node at (5.4,0.5) {.};
  \end{tikzpicture}
\end{equation*}
It remains to remove all subdiagonal elements of the matrices $S_2$ and $S_3$,
this means that all core transformations between matrices $R_1$ and $R_2$, and  the
matrices $R_2$ and $R_3$ need
to be removed. We will first remove all the core transformations acting on rows $1$ and
$2$, followed by those acting on rows $2$ and $3$, and so forth. We remove transformations
from the
right to the left, so first the top core transformation between $R_2$ and $R_3$ is removed
followed
by the top core transformation between $R_1$ and $R_2$.


First we bring the top core transformation in the last sequence to the outer left. 
The
transformation has to undergo two pass-through operations: one with $R_2$ and one with
$R_1$, and two turnovers to get it there
\begin{equation*}
  \centering
  \begin{tikzpicture}[baseline={(current bounding box.center)},scale=1.5,y=-1cm]
  \tikzrotation{-1.6}{0.4}{};
  \tikzrotation{-1.0}{0.2}{};
  \turnoverrl{-1.0}{0.2}{-1.6}{0.4}
  \shiftthroughrl{.4}{-1.0}{0.2} 

  \foreach \j in {0.0,0.2,0.4,...,.8} {
    \tikzrotation{\j-1.6}{\j}
  }      
  \uppertriangular{-1.0}{0.0}{};

  \tikzrotation{.4}{0.2}{};
  \tikzrotation{1.0}{0.0}{};
  \turnoverrl{1.0}{0.0}{.4}{0.2}
  \shiftthroughrl{2.6}{1.0}{0.0} 

  \foreach \j in {0.0,0.2,0.4,...,.8} {
    \tikzrotation{\j+0.6}{\j}
  }      
  \uppertriangular{1.2}{0.0}{};

  \foreach \j in {0,...,4} {
    \tikzrotation{2.6+\j/5}{\j/5}
  }      
  \uppertriangular{2.8}{0.0}{$S_3$};

  \uppertriangular{4.2}{0.0}{$T^{-1}$}; 
  \node at (5.4,.5) {.};
  \end{tikzpicture}  
\end{equation*}

To continue the chasing a similarity transformation is executed removing the rotation from
the left and bringing it to the right of the product. Pictorially 
\begin{equation*}
  \begin{tikzpicture}[baseline={(current bounding box.center)},scale=1.5,y=-1cm]
  \tikzrotation{-.4}{0.4}{};
  \transferbulgelr{-.4}{5.5}{0.4}
  \tikzrotation{5.5}{0.4}{};


  \foreach \j in {0.0,0.2,0.4,...,.8} {
    \tikzrotation{\j-.4}{\j}
  }      
  \uppertriangular{-.2}{0.0}{};


  \foreach \j in {0.0,0.2,0.4,...,.8} {
    \tikzrotation{\j+1.2}{\j}
  }      
  \uppertriangular{1.4}{0.0}{};

  \foreach \j in {1,...,4} {
    \tikzrotation{2.6+\j/5}{\j/5}
  }      
  \uppertriangular{2.8}{0.0}{$S_3$};

  \uppertriangular{4.2}{0.0}{$T^{-1}$}; 
  \node at (5.7 ,.5) {.};
  \end{tikzpicture}
\end{equation*}
As a result we now have a core transformation on the outer right operating on rows $3$ and
$4$. Originally this transformation was acting on rows $1$ and $2$. Every time we do a
turnover the core transformation moves down a row. The operation of moving a core
transformation to the outer left and then bringing it back to the right via a similarity
transformation is vital in all forthcoming algorithms, we will name this a
\emph{sweep}. Depending on the number of turnovers executed in a sweep, the
effect is clearly a downward move of the involved core transformation. Finally it hits the
bottom and gets fused with another core transformation.

We continue this procedure and try to execute another sweep: we
move the transformation on the outer right back to the outer left by $5$ pass-through operations ($3$ are needed
to pass
the transformation through $T^{-1}=T_3^{-1} T_2^{-1} T_1^{-1}$) and $2$
turnovers. At the end, the core transformation under consideration operates on the bottom
two rows as it was moved down $2$ times. 
 The result looks like 
\begin{equation*}
  \centering
  \begin{tikzpicture}[baseline={(current bounding box.center)},scale=1.5,y=-1cm]



  \foreach \j in {0.0,0.2,0.4,...,.8} {
    \tikzrotation{\j-1.6}{\j}
  }      

  \tikzrotation{-.4}{0.8}{};
  \shiftthroughrl{.8}{-.4}{0.8};  

  \uppertriangular{-1.0}{0.0}{};


  \foreach \j in {0.0,0.2,0.4,...,.8} {
    \tikzrotation{\j+.4}{\j}
  }   
   
  \tikzrotation{.8}{0.8}{};
  \turnoverrl{1.4}{0.6}{.8}{0.8}
  \tikzrotation{1.4}{0.6}{};
  \shiftthroughrl{2.4}{1.4}{0.6};  
  \uppertriangular{1.0}{0.0}{};

  \foreach \j in {1,...,4} {
    \tikzrotation{2.2+\j/5}{\j/5}
  }

  \tikzrotation{2.4}{0.6}{};
  \turnoverrl{3.0}{0.4}{2.4}{0.6}
  \tikzrotation{3.0}{0.4}{};
  \shiftthroughrl{4.2}{3.0}{0.4};  
  \uppertriangular{2.8}{0.0}{};

  \tikzrotation{4.2}{0.4}{};
  \shiftthroughrl{5.5}{4.2}{0.4};
  \uppertriangular{4.2}{0.0}{}; 

  \tikzrotation{5.5}{0.4}{};

\node at (5.7,.5) {.};


  \end{tikzpicture}
%
\end{equation*}
At this point it is no longer possible to 
move the core transformation further to the left. We can get rid of it by
fusing it with the bottom core transformation of the first sequence. 
%
 We have removed a single core
transformation and it remains to chase the others in a similar fashion. Pictorially
we have
\begin{equation*}
  \begin{tikzpicture}[baseline={(current bounding box.center)},scale=1.5,y=-1cm]



  \foreach \j in {0.0,0.2,0.4,...,.8} {
    \tikzrotation{\j-1.6}{\j}
  }      

  \tikzrotation{-.4}{0.8}{};
  \shiftthroughrl{-.4}{-.8}{.8}

  \uppertriangular{-1.0}{0.0}{};


  \foreach \j in {0.0,0.2,0.4,...,.8} {
    \tikzrotation{\j+.4}{\j}
  }   
   
  \uppertriangular{1.0}{0.0}{};

  \foreach \j in {1,...,4} {
    \tikzrotation{2.2+\j/5}{\j/5}
  }

  \uppertriangular{2.8}{0.0}{};

  \uppertriangular{4.2}{0.0}{}; 
  \node at (5.4,.5) {.};
  \end{tikzpicture}
\end{equation*}

The top core transformation in the second sequence is now marked for removal.
Pictorially we have accumulated all steps leading to
\begin{equation*}
  \begin{tikzpicture}[baseline={(current bounding box.center)},scale=1.5,y=-1cm]

  \foreach \j in {0.0,0.2,0.4,...,.8} {
    \tikzrotation{\j-1.6}{\j}
  }

  \shiftthroughrl{.4}{-1.2}{0};  
  \tikzrotation{-1.2}{0}{};
  \turnoverrl{-1.2}{0}{-1.8}{.2}

  \tikzrotation{-1.8}{.2}{};
  \transferbulgelr{-1.8}{5.6}{.2}
  \tikzrotation{5.6}{.2}{};

  \shiftthroughrl{5.6}{4.2}{.2};
  \uppertriangular{4.2}{0.0}{}; 
  \tikzrotation{4.2}{.2}{};
  \shiftthroughrl{4.2}{2.8}{.2};  
  \uppertriangular{2.8}{0.0}{};
  \tikzrotation{2.8}{.2}{};

  \turnoverrl{2.8}{.2}{2.2}{.4}
  \tikzrotation{2.2}{.4}{};
  \shiftthroughrl{2.2}{1.2}{0.4} 
  \uppertriangular{1.0}{0.0}{};
  \tikzrotation{1.2}{0.4}{};
  
  \turnoverrl{1.2}{.4}{0.6}{.6}
  \tikzrotation{0.6}{.6}{};
  \shiftthroughrl{.6}{-.6}{.6}
  \uppertriangular{-1.0}{0.0}{};
  \tikzrotation{-0.6}{.6}{};

  \turnoverrl{-.6}{.6}{-1.2}{.8}
  \tikzrotation{-1.2}{.8}{};
  \transferbulgelr{-1.2}{5.6}{.8}
  \tikzrotation{5.6}{.8}{};

  \shiftthroughrl{5.6}{4.2}{.8}
  \tikzrotation{4.2}{.8}{};
  \shiftthroughrl{4.2}{3.4}{.8};  
  \tikzrotation{3.4}{.8}{};
  \shiftthroughrl{3.4}{3.0}{.8};

  \foreach \j in {0.0,0.2,0.4,...,.8} {
    \tikzrotation{\j+.4}{\j}
  }

  \foreach \j in {1,...,4} {
    \tikzrotation{2.2+\j/5}{\j/5}
  }

\node at (5.8,.5) {.};
  \end{tikzpicture}
\end{equation*}

The next core transformation to be chased is the second one of the outer right sequence of core
transformations we are handling. Pictorially all steps of the chasing  at once look like
\begin{equation*}
  \begin{tikzpicture}[baseline={(current bounding box.center)},scale=1.5,y=-1cm]
  \transferbulgelr{-1.4}{5.6}{.6}


  \foreach \j in {0.0,0.2,0.4,...,.8} {
    \tikzrotation{\j-1.6}{\j}
  }


  \tikzrotation{-1.4}{.6}{};
  \turnoverrl{-.8}{.4}{-1.4}{.6}
  \tikzrotation{-.8}{0.4}{};
  \shiftthroughrl{.4}{-.8}{.4};  
  \uppertriangular{-1.0}{0.0}{};

  \tikzrotation{.4}{0.4}{};
  \turnoverrl{1.0}{0.2}{.4}{0.4}
  \tikzrotation{1.0}{0.2}{};
  \shiftthroughrl{2.4}{1.0}{0.2} 

  \foreach \j in {0.2,0.4,...,.8} {
    \tikzrotation{\j+.4}{\j}
  }

  \shiftthroughrl{1.6}{1.2}{.8}
  \tikzrotation{1.6}{.8}{};
  \shiftthroughrl{2.6}{1.6}{0.8};  
  \uppertriangular{1.0}{0.0}{};

  \foreach \j in {1,...,4} {
    \tikzrotation{2.2+\j/5}{\j/5}
  }

  \tikzrotation{2.6}{.8}{};
  \turnoverrl{3.2}{.6}{2.6}{.8}
  \tikzrotation{3.2}{.6}{};
  \shiftthroughrl{4.6}{3.2}{.6};  
  \uppertriangular{2.8}{0.0}{};

  \tikzrotation{4.6}{.6}{};
  \shiftthroughrl{5.6}{4.6}{.6};
  \uppertriangular{4.2}{0.0}{}; 

  \tikzrotation{5.6}{.6}{};
\node at (5.8,.5) {.};
  \end{tikzpicture}
\end{equation*}
The entire procedure to chase a single core transformation consists of executing
sweeps until the core transformation hits the bottom and can fuse with another core transformation.





At the very end
we obtain the
 factorization 
\begin{equation}
  \label{fig:red:9}
    \begin{tikzpicture}[baseline={(current bounding box.center)},scale=1.5,y=-1cm]
  \foreach \j in {0.0,0.2,0.4,...,.8} {
    \tikzrotation{\j-0.2}{\j}
  }      
  \uppertriangular{0}{0.0}{};
  \drawbrace{-0.2}{1.0}{1.1}{$\tilde{S}_1$}

  \uppertriangular{1.4}{0.0}{$\tilde{S}_2$};

  \uppertriangular{2.8}{0.0}{$\tilde{S}_3$};

  \uppertriangular{4.2}{0.0}{$\tilde{T}_{3}^{-1}$}; 
  \uppertriangular{5.6}{0.0}{$\tilde{T}_{2}^{-1}$}; 
  \uppertriangular{7.0}{0.0}{$\tilde{T}_{1}^{-1}$}; 
\node at (8.2,.5) {.};
  \end{tikzpicture} 
\end{equation}
This matrix is in upper Hessenberg form and on this factorization we will run
the product eigenvalue problem.

The algorithm annihilates the unwanted core transformations acting on the first rows, followed
    by those acting on the second row, and so forth. As a consequence a single full sweep
    from right to left always takes $k$
    turnovers and $2k$ pass-through operations, and gets thus
    a complexity $\mathcal{O}(k)$. To get
    rid of a single core transformation $G_i$ we need approximately
    $\lfloor(dk-i)/k\rfloor$ sweeps leading to an approximate complexity count
    \begin{equation*}
      \sum_{i=1}^{dk} \sum_{j=1}^{k-1} \frac{dk-i}{k} k = \mathcal{O}(d^2 k^3), 
    \end{equation*}
    where $i$ runs over the core transformations and $j$ over the sequences.

\section{Product eigenvalue problem}
\label{sec:pep}

The following discussion is a concise description of the actual QR algorithm. It is based on the results of Aurentz, Mach, Vandebril, and Watkins
\cite{AuMaVaWa15,VaWa12} combined with Watkins's interpretation of product eigenvalue
problems \cite{q994}.
For simplicity, we will only describe a single shifted QR step in the Hessenberg case, for
information beyond the Hessenberg case we refer to Vandebril \cite{Va10d}.

Suppose we have a Hessenberg-triangular pencil
$(\tilde{S},\tilde{T})$, with $\tilde{S}=\tilde{S}_1 \cdots
\tilde{S}_k$ and $\tilde{T}=\tilde{T}_1 \cdots \tilde{T}_k$, where  $\tilde{S}_1$ is of
Hessenberg form and all other factors upper triangular and nonsingular. The algorithm for solving the
generalized product eigenvalue problem can be seen as a QR algorithm applied on
$\tilde{S}\tilde{T}^{-1}$. Pictorially, for $d=2$, $k=3$, and $n=dk=6$, this looks as in
\eqref{fig:red:9}. As before all upper
triangular factors are combined into a single one, where
$\tilde{S}_1=\tilde{Q}_1 \tilde{R}_1$ is a QR factorization of $\tilde{S}_1$. Pictorially
we get
\begin{center}
  \begin{tikzpicture}[scale=1.66,y=-1cm]
    \node[above] at (-1.2,.6) {$\tilde{S}\tilde{T}^{-1}=$};
    \foreach \j in {0.0,0.2,0.4,...,.8} {
      \tikzrotation{\j-0.4}{\j}
    }      
    \drawbrace{-0.4}{0.4}{1.1}{$\tilde{Q}$}
    \uppertriangular{1.2}{0.0}{}
    \drawbrace{1.2}{2.2}{1.1}{$\tilde{R}_1\tilde{S}_2\tilde{S}_3\tilde{T}^{-1}_3\tilde{T}^{-1}_2\tilde{T}^{-1}_1$}
    %
    \node[above] at (2.5,.6) {.};      
  \end{tikzpicture}
\end{center}

To initiate the core chasing algorithm we pick a 
suitable shift $\mu$ and form $u = (\tilde{S} - \mu \tilde{T}) e_{1}$. 
The initial similarity transformation is determined by the core transformation $U_1$ such
that $U_{1}^* u = \alpha e_{1}$ for some $\alpha$. 
We fuse the two outer left core transformations and pass the core transformation 
 $U_1$ on the
right  through the upper triangular matrix to get a new core transformation
$X_1$. Pictorially we get the left of \eqref{eq:first_bulge}. The resulting matrix is not of upper Hessenberg form anymore, it is perturbed by the
core transformation $X_1$. We will chase this core transformation to the bottom. We name
this core transformation $X_1$ the \emph{misfit}. A turnover will move $X_1$ to the outer
left. Pictorially we get the right of \eqref{eq:first_bulge}.

\begin{equation}\label{eq:first_bulge}
  \begin{tikzpicture}[baseline={(current bounding box.center)},scale=1.66,y=-1cm]
    \tikzrotation[black][][$U_{1}^*$]{-.4}{0}
     \draw[->] (-.35,0.1) -- (-.1,0.1);  

    \foreach \j in {0.0,0.2,0.4,...,.8} {
      \tikzrotation{\j}{\j}
    }  
    \uppertriangular{.6}{0.0}{}
    \node[above] at (2.3,.6) {};

    \pgfmathsetmacro{\xuone}{-.3}
    \pgfmathsetmacro{\xutwo}{0.4}
    \pgfmathsetmacro{\xuthree}{3.2}
    \pgfmathsetmacro{\xufour}{3.9}
    \pgfmathsetmacro{\xufive}{5.9}

    \tikzrotation{\xutwo}{0}
    \node at (\xutwo,-0.2) [align=center] {$X_{1}$};
    \shiftthroughrl{2.0}{\xutwo}{0.0}
    \tikzrotation[black][][$U_{1}$]{2.0}{0}

  \end{tikzpicture}
\qquad\qquad\quad
  \begin{tikzpicture}[baseline={(current bounding box.center)},scale=1.66,y=-1cm]
    \foreach \j in {0.0,0.2,0.4,...,.8} {
      \tikzrotation{\j}{\j}
    }  
    \uppertriangular{.6}{0.0}{}
   \node[above] at (2.3,.6) {$$};
    
    \pgfmathsetmacro{\xuone}{-.2}
    \pgfmathsetmacro{\xutwo}{0.4}
    \pgfmathsetmacro{\xuthree}{3.9}
    
    \tikzrotation[black][][][$X_{1}$]{\xutwo}{0}
    \tikzrotation[black][][${U}_{2}$]{\xuone}{0.2}

    \turnoverrl{\xutwo}{0.0}{\xuone}{0.2}
    

  \end{tikzpicture}
\end{equation}
Next we execute a similarity transformation with $U_2$. This will cancel out $U_2$ on the
left and bring it to the right. Next pass
$U_{2}$ through the upper triangular matrix. Pictorially the flow looks like the left of \eqref{eq:ch:2}.
This looks similar to \eqref{eq:first_bulge}, except for the misfit
 which has moved downward one row.  We can continue now by executing a turnover, a
 similarity, and a pass-through to move the misfit down one more position. Pictorially we
 end up in the right of \eqref{eq:ch:2}.
\begin{equation}
\label{eq:ch:2}
  \begin{tikzpicture}[baseline={(current bounding box.center)},scale=1.66,y=-1cm]
    \foreach \j in {0.0,0.2,0.4,...,.8} { \tikzrotation{\j}{\j} }
    \tikzrotation[black][][$U_2$]{-.2}{.2}{};
    \transferbulgelr{-.2}{2.2}{.2}    
    \uppertriangular{.8}{0.0}{}

    \pgfmathsetmacro{\xuone}{-.3} 
    \pgfmathsetmacro{\xutwo}{0.6}
    \pgfmathsetmacro{\xuthree}{3.2} 
    \pgfmathsetmacro{\xufour}{3.9}
    \pgfmathsetmacro{\xufive}{5.9}

    \tikzrotation{\xutwo}{0.2} 
    \node at (\xutwo,0.0) [align=center] {$X_{2}$};
    \shiftthroughrl{2.2}{\xutwo}{0.2} 
    \tikzrotation[black][][$U_{2}$]{2.2}{0.2}

  \end{tikzpicture}
\qquad\qquad\quad
  \begin{tikzpicture}[baseline={(current bounding box.center)},scale=1.66,y=-1cm]
    \foreach \j in {0.0,0.2,0.4,...,.8} {
      \tikzrotation{\j}{\j}
    }  
    \uppertriangular{1.2}{0.0}{}
    
    \pgfmathsetmacro{\xuone}{0.0}
    \pgfmathsetmacro{\xutwo}{0.6}
    \pgfmathsetmacro{\xuthree}{3.5}
    \pgfmathsetmacro{\xufour}{.8}
    
    \tikzrotation[black][][][$X_{2}$]{\xutwo}{0.2}
    \tikzrotation[black][][][\;$X_{3}$]{\xufour}{0.4}
    \tikzrotation[black][][${U}_{3}$]{2.6}{0.4}

    \turnoverrl{\xutwo}{0.2}{\xuone}{0.4}
    
    \tikzrotation[black][][$U_{3}$]{-.0}{0.4}
    \transferbulgelr{\xuone}{2.6}{0.4}
    \shiftthroughrl{2.6}{\xufour}{0.4} 

  \end{tikzpicture}
\end{equation}

After $n-2$ similarities we are not able to execute a turnover anymore. The final core
transformation $X_{n-1}$ fuses with the last core transformation of the sequence and we are done
\begin{center}
  \begin{tikzpicture}[scale=1.66,y=-1cm]
    \foreach \j in {0.0,0.2,0.4,...,.8} {
      \tikzrotation{\j+.4}{\j}
    }  
    \uppertriangular{1.4}{0.0}{}

    \pgfmathsetmacro{\xuone}{1.6}
    \tikzrotation[black][][][$X_{n-1}$]{\xuone}{.8}
    \path[->,out=180,in=0] (\xuone-0.1,.9) edge (\xuone-0.35,.9);
\node at (2.6,.5) {.};
  \end{tikzpicture}
\end{center}

In fact 
we continue executing sweeps until the core transformation hits the bottom and gets fused.
After few QR steps a deflation will occur. Classically a deflation in a Hessenberg matrix
is signaled by subdiagonal elements being relatively small with respect to the
neighbouring diagonal elements. Before being able to utilize this convergence criterion we
would need to compute and accumulate all diagonal elements of the compactly stored upper
triangular factors.
In this factored form, however, a cheaper and more reliable \cite{MaVa14} criterion is more suitable.
Deflations are signaled by almost diagonal
core transformations in the descending sequence preceding the upper triangular factors.
Only a simple check of the core transformations is required and we do not need to extract the diagonal elements out the compact
representation, nor do we need to accumulate them.

Solving the actual product eigenvalue problem requires to compute $n=dk$ eigenvalues,
  where on average each eigenvalue should be found in a few QR steps. A single QR step
  requires the chasing of an artificially introduced core transformation. During each sweep this
  core transformation moves down a row because of a single turnover, hence $n=dk$ sweeps are
  required, each taking $2k$
  pass-throughs and one turnover operations. In total this amounts to $\mathcal{O}(d^2 k ^3)$.

\section{Removing infinite and zero eigenvalues}
\label{sec:deflation}

Typically the unitary-plus-spike matrices in the factorizations will be nonsingular,
but there are exceptions.  If the pencil has a zero eigenvalue, then $S$ 
will be singular, and therefore one of the factors $R_{i}$
will necessarily be singular.  If there is an infinite eigenvalue, $T$ will be singular, 
so one of the factors $T_{i}$ must be singular.  
We must therefore ask whether singularity of any of these factors can cause any
difficulties.  

As we shall see, infinite eigenvalues present no problems.
They are handled automatically by our algorithm.  Unfortunately we cannot say the 
same for zero eigenvalues.   For the proper functioning of the Francis QR iterations, 
any exactly zero eigenvalues must be detected and deflated out beforehand.  We 
will present a procedure for doing this.  

\subsection{Singular unitary-plus-rank-one matrices}

We begin by characterizing singularity of a unitary-plus-spike factor.  
Suppose $R=R_i$ \eqref{eq:ri}  is singular.  (The same considerations apply to $T_{i}$.)  
There must be a zero on the main diagonal, and this can occur only at the intersection of the diagonal
and the spike, i.e.\ at position $(\ell,\ell)$, where $\ell = n + 1 - i$.  $R$ is stored in the 
factored form $R = P^{T}C_{n}^{*}\cdots C_{1}^{*}(B_{1}\ldots B_{n} + e_{1}y^{T})P$.
The validity of this representation was established in \cite{AuMaVaWa15}, and it 
remains valid even though $R$ is singular.  Let's see how singularity shows 
up in the representation.  

Recall from the construction of $\mathcal{C}$ and $\mathcal{B}$ that
$C_{\ell+1},\ldots,C_n$, all have active part
$\left[\begin{smallmatrix} 0 & 1\\ 1 & 0\end{smallmatrix}\right]$.  In other words  
$C_{j} = F_{j}$ for $j=\ell + 1$, \ldots, $n$.  But now the additional  zero element 
at position $\ell$ in $\underline{x}$ implies that also $C_{\ell} = F_{\ell}$ as well.
But then $B_{\ell} = C_{\ell}F_{\ell} = F_{\ell}^{2} = I$.
The converse holds as well:
$B_{\ell}$ is trivial if and only if $R$ has a zero at the $\ell$th diagonal position.  
All of the other $B_{j}$ are equal to $C_{j}$, and these are all nontrivial.



This is the situation at the time of the initial construction of $R$.  In the course of the 
reduction algorithm and subsequent Francis iterations, $R$ is modified repeatedly by 
having core transformations passed through it, but it continues to be singular and
it continues to contain one trivial $B_{i}$ core transformation, as 
Theorem~\ref{thm:rtriangular} below demonstrates.    As preparation for this theorem
we remark that Theorems~4.2, 4.3, and 4.7 of \cite{AuMaVaWa15} remain valid, 
ensuring that all of the $C_{i}$ remain nontrivial.

\begin{theorem}[Modification of \protect{\cite[Theorem~4.3]{AuMaVaWa15}}]
\label{thm:rtriangular}
  Consider a factored unitary-plus-rank-one matrix $R = P^{T}\mathcal{C}^{*}(\mathcal{B} + e_{1}y^{T})P = P^{T}C_{n}^{*}\cdots C_{1}^{*}(B_{1}\cdots
  B_{n} + e_{1}y^{T})P$, with core transformations $C_{1}$, \ldots, $C_{n}$
  nontrivial.  Then $R$ is upper triangular.  Moreover $B_\ell$ is trivial if and only if the 
  $\ell$th  diagonal element of $R$ equals zero. 
\end{theorem}

\begin{proof}
Originally we started out with 
\begin{equation}\label{rform}
 \underline{R} = \left[\begin{array}{cc}
R & \times \\ 0 & 0
\end{array}\right]
\end{equation}
where the symbol $\times$ represents a vector that is not of immediate interest.  
The bottom row of $\underline{R}$ is zero initially and it remains zero forever.  This is so because the core
transformations that are passed into and out of $R$ act on rows and columns $1$, \ldots, $n$;
they do not alter row $n+1$.     
Letting $H = B + e_{1}y^{T}$, we have $\underline{R} = \mathcal{C}^{*}H$, or equivalently 
$H = \mathcal{C}\underline{R}$.  Partition this equation as 
\begin{equation}\label{eq:hcrpart}
\left[\begin{array}{cc}
\times & \times \\  \tilde{H} & \times 
\end{array}\right] =
\left[\begin{array}{cc}
\times & \times \\  \tilde{C} & \times 
\end{array}\right] 
 \left[\begin{array}{cc}
R & \times \\ 0 & 0
\end{array}\right],
\end{equation}
where $\tilde{H}$ and $\tilde{C}$ are $n \times n$ and the vectors marked $\times$ are not of
immediate interest.  We deduce that 
$$\tilde{H} = \tilde{C}R.$$
$H$ is upper Hessenberg, so $\tilde{H}$ is upper triangular.  
Since all $C_{i}$ are nontrivial, $\mathcal{C}$ is a proper
upper Hessenberg matrix ($c_{i+1,i} \neq 0$, $i=1,\ldots,n$), so 
 $\tilde{C}$ is upper triangular and nonsingular.  We have
\begin{equation}\label{rstaystriangular}
R = \tilde{C}^{-1}\tilde{H}.
\end{equation} 
so $R$ must be upper triangular.  Now looking at the main diagonal  of the equation
$\tilde{H} = \tilde{C}R$, we find that $h_{\ell+1,\ell} = c_{\ell+1,\ell}\, r_{\ell,\ell}$.   
Since $c_{\ell+1,\ell} \neq 0$ we see that $r_{\ell,\ell} = 0$ if and only if $h_{\ell+1,\ell} = 0$,
and this happens if and only if $B_{\ell}$ is trivial.  
\end{proof} 

The existence of trivial core transformations in the factors presents no difficulties for the
reduction to Hessenberg-triangular form.  In some cases it will result in trivial core transformations
being chased forward, but this does no harm.  Now let us consider what happens in the iterative phase
of the procedure.  

\subsection{Infinite eigenvalues}

Behavior of infinite eigenvalues under Francis iterations is discussed in \cite{b333}.  There
it is shown that a zero on the main diagonal of $T$ gets moved up by one position on each
iteration.   Let's see how this manifests itself in our structured case.   Consider an example
of a singular $T_{i}$ with a trivial core transformation  in the third position, $B_{3} = I$, as shown
pictorially below.    
Since $T_{i}$
is in the ``inverted'' part, core transformations pass through it from left to right.   
Suppose we pass $G_{2}$ through $T_{i}$, transforming $G_{2}T_{i}$ to $\tilde{T}_{i}\tilde{G}_{2}$.  
The first turnover is routine, but in the second
turnover there is a trivial factor:  $F_{3}B_{2}B_{3} = F_{3}B_{2}I$.  This turnover is thus
trivial: $F_{3}B_{2}I = IF_{3}B_{2}$.  $I$ becomes the new  $B_{2}$.  The old $B_{2}$ is pushed
out of the sequence to become $\tilde{G}_{2}$.  Pictorially
\begin{equation*}
    \begin{tikzpicture}[baseline={(current bounding box.center)},scale=1.66,y=-1cm]




     \node at (4.0,0.2) [above] {$G_2$};
     \tikzrotation[black][][]{4.0}{0.2};      
     \turnoverlr{4.0}{0.2}{4.6}{0.4}
     \tikzrotation[black][][]{4.6}{0.4};      
     \shiftthroughlr{4.6}{5.0}{0.4};
     \tikzrotation[black][][]{5.0}{0.4};      

     \node[above] at (5.23,.2) {$B_2$};
     \node[below] at (5.0,.64) {$F_3$};

    \foreach \j in {0,1}{
          \tikzrotation{\j/5+5.0}{\j/5}
    }      

    \foreach \j in {3,4}{
          \tikzrotation{\j/5+5.0}{\j/5}
    }

    \foreach \j in {0,1,...,4}{
          \tikzrotation{\j/5+3.8}{-\j/5+.8}
    }

    \begin{scope}[xshift=4cm]
      \node at (6.0,0.2) [above] {$\tilde G_2$}; 
      \tikzrotation[black][][]{5.6}{0.2};
      \tikzrotation[black][][]{6.0}{0.2};
     \node[above] at (5.63,.2) {$B_2$};
      \shiftthroughlr{5.6}{6.0}{0.2}; 

      \foreach \j in {0}{ \tikzrotation{\j/5+5.0}{\j/5} }

      \foreach \j in {2,3,4}{ \tikzrotation{\j/5+5.0}{\j/5} }

      \foreach \j in {0,1,...,4}{ \tikzrotation{\j/5+3.8}{-\j/5+.8}};
\node[above] at (6.4,.7) {.};
    \end{scope}





      \node[above] at (6.8,.7) {becomes};
    \end{tikzpicture}
\end{equation*}
The trivial core transformation in $\mathcal B$ has moved up one position. 

On each iteration it moves up one position  until it gets to the top.  
At that point an infinite eigenvalue can be deflated at the top.  
The deflation happens automatically;  no special action is necessary.  

\subsection{Zero eigenvalues}

In the case of a zero eigenvalue, one of the $R_{i}$ factors has a trivial
core transformation, for example, $B_{3}$.   During the iterations, core
transformations pass through $R_{i}$ from right to left.  One might hope
that the trivial core transformation gets pushed downward by one position
on each iteration, eventually resulting in a deflation at the bottom.  Unfortunately
this is not what happens.  When a transformation $G_{2}$ is pushed into 
$R_{i}$ from the right, the first turnover is trivial:  $B_{2} I G_{2} = I \tilde{B}_{2} I$, 
where $\tilde{B}_{2} = B_{2}G_{2}$.  A trivial core transformation is ejected
on the left, and the iteration dies.

It turns out that this problem is not caused by the special structure of our factors
but by the fact that we are storing the matrix in $QR$-decomposed form.  There
is a simple general remedy.   Consider a singular upper-Hessenberg $A$ with no special 
structure, stored in the form $A = Q_{1} \cdots Q_{n-1}R$.   If $A$ is
singular, then so is $R$, so $r_{ii}=0$ for some $i$.  The reader can easily check that if $r_{ii} = 0$ for $i<n$,
then $a_{i+1,i} = 0$, so $A$ is not properly upper Hessenberg, and a deflation should be possible.  We will
show how to do this below, but first consider the case when $r_{nn}=0$.  
We have 
\begin{displaymath}
A \ = \ QR \ = \ \parbox{4.7cm}{
\begin{tikzpicture}[scale=1.66,y=-1cm]
\phantom{\tikzrotation{0}{0}}
\foreach \j in {0,...,4}{\tikzrotation{\j/5-1.25}{\j/5}}
\draw (-.15,-.1) -- (-.2,-.1) -- (-.2,1.1) -- (-.15,1.1);
\draw (1.15,-.1) -- (1.2,-.1) -- (1.2,1.1) -- (1.15,1.1);
\foreach \j in {0,...,4}{
   \foreach \i in {\j,...,5}{\node at (\i/5,\j/5)
     [align=center,scale=1.0]{$\times$};}}
\end{tikzpicture}
},
\end{displaymath}
where $R$ has a zero at the bottom.   Do a similarity transformation that moves
the entire matrix $Q$ to the right.  Then pass the core transformations back through $R$, and notice
that when $Q_{n-1}$ is multiplied into $R$, it acts on columns $n-1$ and $n$ and does not create 
a bulge.  Thus nothing comes out on the left.  Or, if you prefer, we can say that $\hat{Q}_{n-1}=I$ comes out
on the left:
\begin{displaymath}
\parbox{4.4cm}{
\begin{tikzpicture}[scale=1.66,y=-1cm]
\draw (-.15,-.1) -- (-.2,-.1) -- (-.2,1.1) -- (-.15,1.1);
\draw (1.15,-.1) -- (1.2,-.1) -- (1.2,1.1) -- (1.15,1.1);
\foreach \j in {0,...,4}{
   \foreach \i in {\j,...,5}{\node at (\i/5,\j/5)
     [align=center,scale=1.0]{$\times$};}}
\foreach \j in {0,...,4}{\tikzrotation{\j/5+1.45}{\j/5}} 
\phantom{\tikzrotation{.5}{.8}}
\end{tikzpicture}
} \ = \ \parbox{4.7cm}{
\begin{tikzpicture}[scale=1.66,y=-1cm]
\foreach \j in {0,...,3}{\tikzrotation{\j/5-1.25}{\j/5}}
\phantom{\tikzrotation{-.45}{.8}}
\draw (-.15,-.1) -- (-.2,-.1) -- (-.2,1.1) -- (-.15,1.1);
\draw (1.15,-.1) -- (1.2,-.1) -- (1.2,1.1) -- (1.15,1.1);
\foreach \j in {0,...,4}{
   \foreach \i in {\j,...,5}{\node at (\i/5,\j/5)
     [align=center,scale=1.0]{$\times$};}}
\end{tikzpicture}
}.
\end{displaymath}
Now we can deflate out the zero eigenvalue.

It is easy to relate what we have done here to established theory.  It is well known \cite{b333} that
if there is a zero eigenvalue, one step of the explicit QR algorithm with zero shift ($A = QR$, $RQ= \hat{A}$) 
will extract it.  This is exactly what we have done here.  One can equally well check that if one does a single Francis  
iteration with shift  $\rho = 0$, this is exactly what results.  

Now consider the case $r_{ii}=0$, $i<n$.  Depicting the case $i=3$ we have 
\begin{displaymath}
A \ = \ QR \ = \ \parbox{4.7cm}{
\begin{tikzpicture}[scale=1.66,y=-1cm]
\phantom{\tikzrotation{0}{0}}
\foreach \j in {0,...,4}{\tikzrotation{\j/5-1.25}{\j/5}}
\draw (-.15,-.1) -- (-.2,-.1) -- (-.2,1.1) -- (-.15,1.1);
\draw (1.15,-.1) -- (1.2,-.1) -- (1.2,1.1) -- (1.15,1.1);
\foreach \j in {0,...,4}{
   \foreach \i in {\j,...,4}{\node at (\i/5+.2,\j/5)
     [align=center,scale=1.0]{$\times$};}}
\foreach \j in {0,...,1}{\node at (\j/5,\j/5) [align=center,scale=1.0]{$\times$};}
\foreach \j in {3,...,5}{\node at (\j/5,\j/5) [align=center,scale=1.0]{$\times$};}
\end{tikzpicture}
}.
\end{displaymath}
Pass $Q_{i+1}$, \ldots, $Q_{n-1}$ from left to right through $R$ to get them out of
the way:
\begin{displaymath}
\parbox{5.3cm}{
\begin{tikzpicture}[scale=1.66,y=-1cm]
\phantom{\tikzrotation{0}{0}}
\foreach \j in {0,...,2}{\tikzrotation{\j/5-1.25}{\j/5}}
\foreach \j in {3,4}{\tikzrotation[gray]{\j/5-1.25}{\j/5}}
\foreach \j in {3,4}{\shiftthroughlr[gray]{\j/5-1.25}{\j/5+.9}{\j/5}}
\foreach \j in {3,4}{\tikzrotation[black]{\j/5+.9}{\j/5}}
\draw (-.15,-.1) -- (-.2,-.1) -- (-.2,1.1) -- (-.15,1.1);
\draw (1.15,-.1) -- (1.2,-.1) -- (1.2,1.1) -- (1.15,1.1);
\foreach \j in {0,...,4}{
   \foreach \i in {\j,...,4}{\node at (\i/5+.2,\j/5)
     [align=center,scale=1.0]{$\times$};}}
\foreach \j in {0,...,1}{\node at (\j/5,\j/5) [align=center,scale=1.0]{$\times$};}
\foreach \j in {3,...,5}{\node at (\j/5,\j/5) [align=center,scale=1.0]{$\times$};}
\end{tikzpicture}
}.
\end{displaymath}
Since these do not touch row or column $i$, the zero at $r_{ii}$ is preserved.  
Now the way is clear to multiply $Q_{i}$ into $R$ without creating a bulge.  
This gets rid of $Q_{i}$.  Now the core transformations that were passed through $R$
can be returned to their initial positions either by passing them back through $R$ or 
by doing a similarity transformation.  Either way the result is 
\begin{displaymath}
\parbox{4.7cm}{
\begin{tikzpicture}[scale=1.66,y=-1cm]
\phantom{\tikzrotation{0}{0}}
\foreach \j in {0,...,1}{\tikzrotation{\j/5-1.25}{\j/5}}
\foreach \j in {3,...,4}{\tikzrotation{\j/5-1.25}{\j/5}}
\draw (-.15,-.1) -- (-.2,-.1) -- (-.2,1.1) -- (-.15,1.1);
\draw (1.15,-.1) -- (1.2,-.1) -- (1.2,1.1) -- (1.15,1.1);
\foreach \j in {0,...,4}{
   \foreach \i in {\j,...,4}{\node at (\i/5+.2,\j/5)
     [align=center,scale=1.0]{$\times$};}}
\foreach \j in {0,...,1}{\node at (\j/5,\j/5) [align=center,scale=1.0]{$\times$};}
\foreach \j in {3,...,5}{\node at (\j/5,\j/5) [align=center,scale=1.0]{$\times$};}
\end{tikzpicture}
},
\end{displaymath}
or more compactly
\begin{displaymath}
\parbox{3.8cm}{
\begin{tikzpicture}[scale=1.66,y=-1cm]
\phantom{\tikzrotation{0}{0}}
\foreach \j in {0,...,1}{\tikzrotation{\j/5-.65}{\j/5}}
\foreach \j in {3,...,4}{\tikzrotation{\j/5-1.25}{\j/5}}
\draw (-.15,-.1) -- (-.2,-.1) -- (-.2,1.1) -- (-.15,1.1);
\draw (1.15,-.1) -- (1.2,-.1) -- (1.2,1.1) -- (1.15,1.1);
\foreach \j in {0,...,4}{
   \foreach \i in {\j,...,4}{\node at (\i/5+.2,\j/5)
     [align=center,scale=1.0]{$\times$};}}
\foreach \j in {0,...,1}{\node at (\j/5,\j/5) [align=center,scale=1.0]{$\times$};}
\foreach \j in {3,...,5}{\node at (\j/5,\j/5) [align=center,scale=1.0]{$\times$};}
\end{tikzpicture}
}.
\end{displaymath}
 The eigenvalue problem has been decoupled into two smaller eigenvalue problems,
 one of size $i \times i$ at the top and one of size $(n-i) \times (n-i)$ at the bottom.  
 The upper problem has a zero eigenvalue, which can be deflated immediately by 
 doing a similarity transformation with  
 $Q_{1} \cdots Q_{i-1}$ and passing those core transformations back through $R$ 
 as explained earlier.  The result is  
 \begin{displaymath}
\parbox{3.8cm}{
\begin{tikzpicture}[scale=1.66,y=-1cm]
\phantom{\tikzrotation{0}{0}}
\foreach \j in {0,...,0}{\tikzrotation{\j/5-.65}{\j/5}}
\foreach \j in {3,...,4}{\tikzrotation{\j/5-1.25}{\j/5}}
\draw (-.15,-.1) -- (-.2,-.1) -- (-.2,1.1) -- (-.15,1.1);
\draw (1.15,-.1) -- (1.2,-.1) -- (1.2,1.1) -- (1.15,1.1);
\foreach \j in {0,...,4}{
   \foreach \i in {\j,...,4}{\node at (\i/5+.2,\j/5)
     [align=center,scale=1.0]{$\times$};}}
\foreach \j in {0,...,1}{\node at (\j/5,\j/5) [align=center,scale=1.0]{$\times$};}
\foreach \j in {3,...,5}{\node at (\j/5,\j/5) [align=center,scale=1.0]{$\times$};}
\end{tikzpicture}
}.
\end{displaymath}
We now have an eigenvalue problem of size $(i-1) \times (i-1)$ at the top, a deflated zero
eigenvalue in position $(i,i)$, and an eigenvalue problem of size $(n-i) \times (n-i)$ at the
bottom.

We have described the deflation procedure in the unstructured case, but it can be 
applied equally well in the structured context of this paper.  Instead of a simple upper-triangular
$R$, we have a more complicated $RT^{-1}$, which is itself a product of many factors.  
The implementation details are different, but the procedure is the same.  

\section{Computation of eigenvectors}
\label{sec:eigenvectors}

In practice typically only the left or the right eigenvectors are required. We will
compute left eigenvectors since these ones are easier to retrieve than the right ones.
 If the right eigenvectors are wanted, one can compute the left ones of
 $P(\lambda)^T$.

For $w$ as a left eigenvector corresponding to the eigenvalue $\lambda$, i.e., $w^T
P(\lambda)=0$, we get that      
\[
\hat w =
    \begin{bmatrix}
      w \\
      \lambda w \\
      \vdots \\
      \lambda^{d-1} w \\
    \end{bmatrix}, 
\]
will be a left eigenvector of the companion pencil \eqref{eq:comp_pencil}. This implies that once an
eigenvector $\hat w$ of the companion pencil is computed, the first
$k$ elements of that vector define the eigenvector $w$ of the matrix polynomial. To save
storage and computational cost it suffices thus to compute only the first $k$ elements of each
eigenvector. For reasons of numerical stability, however, we will also compute the last
$k$ elements of $\hat w$ and use its top $k$ elements if $|\lambda|\leq 1$ and its trailing $k$ elements if
$|\lambda|>1$  to define $w$.

Suppose our algorithm has run to completion and we have ended up with the Schur form $U^* (S
 -\lambda  T )V = \hat{S} -\lambda \hat{T}$, where both $\hat{S}$ and $\hat{T}$ are
upper triangular. The left eigenvector, corresponding to the eigenvalue $\hat \lambda$
found in the lower right corner of $\hat{S}-\lambda \hat{T}$ equals $\hat{w}=U e_n$. But
since    the top or bottom $k$ elements of $\hat{w}$ suffice to retrieve $w$ we only need
to store the top $k$ and bottom $k$ rows of $U$. Let $P$ be of size $2k \times dk$ 
\begin{eqnarray*}
  P=
\left[
\begin{array}{ccccc}
I_k & 0 & \ldots & 0 & 0 \\
0 & 0 & \ldots & 0 & I_k
\end{array}
\right],
\end{eqnarray*}
then we see that $PUe_n$ provides us all essential information. As the matrix $U^*$ is an
accumulation of all core transformations applied to the left of $S$ and $T$
during the algorithm we can save computations  by forming
$PU$ directly instead of $U$. 

Let us estimate the cost of forming $PU$. Applying a single core transformation from the right to a
$2k \times dk$ matrix costs $\mathcal{O}(k)$ operations. Each similarity transformation with 
 a core transformation requires us to update $PU$;  it remains thus to count the total
number of similarities executed. 
In the initial reduction procedure to Hessenberg-triangular form we need
at most $d$ 
similarities to remove a single core
transformation. 
As roughly $dk$ core transformations need to be removed, we end up with
$\mathcal{O}(d^2k)$ similarities.
Under the assumption that a $\mathcal{O}(1)$ of QZ steps are required to get convergence to a
single eigenvalue we have $\mathcal{O}(dk)$ similarities for a single eigenvalue, in total
this amounts  $\mathcal{O}(d^2k^2)$ similarities for all eigenvalues.  
In total forming the matrix $PU$ has a complexity of $\mathcal{O}(d^2k^3)$.

Unfortunately the Schur form allows us only to compute the left eigenvector corresponding
to the eigenvalue found in the lower right corner. To compute other eigenvectors we need to
reorder the Schur form so that each corresponding eigenvalue appears once at the bottom, 
after which we
can extract the corresponding eigenvector. To reorder the Schur pencil we can rely on
classical reordering methods \cite{b333}. In our setting we will only swap two eigenvalues
at once and we will use core transformations to do so. To compute the core transformation
that swaps two eigenvalues in the Schur pencil, we need the diagonal and superdiagonal
elements of $\hat{S}$ and $\hat{T}$, these are obtained by computing the diagonal and
superdiagonal elements of each of the involved factors. We refer to Aurentz et al.
\cite{AuMaVaWa15} for details on computing diagonal and superdiagonal elements of a
properly stored unitary-plus-rank-one matrix. After this core transformation is computed
we apply it the Schur pencil and update all involved factors by chasing the core
transformation through the entire sequence. After the core transformation has reached the
other end it is accumulated in $PU$.

If all eigenvectors are required, we need quite some swaps and updating. Let us estimate
the cost. Bringing the eigenvalue in the bottom right position to the upper left top
thereby moving down all other eigenvalues a single time requires $dk-1$ swaps. Doing this
$dk-1$ times makes sure that each eigenvalue has reached the bottom right corner once,
enabling us to extract the corresponding eigenvector. In total $\mathcal{O}(d^2k^2)$ swaps
are thus sufficient to get all eigenvectors.  A single swap requires updating the $2k$
upper triangular factors involving $4k$ turnovers.  Also $PU$ needs to be updated and this
takes $\mathcal{O}(k)$ operations as well. In total this leads to an overall complexity of
$\mathcal{O}(d^2k^3)$ for computing all the eigenvectors of the matrix polynomial.

\section{Backward stability}
\label{sec:backward}
The algorithm consists of three main steps: a preprocessing of the
matrix coefficients, the reduction to Hessenberg-triangular
form, and the actual eigenvalue computations. 
We will quantify how the backward errors can accumulate in all steps. The second
and third step are dealt with simultaneously. In this section we use $\doteq$
to denote an equality where some  second or higher order terms have dropped, $\delta X$
will denote a perturbation of $X$, and $\lesssim$ stands for less than, up to
multiplication with 
polynomial in $d$ and $k$ of modest degree. For simplicity we assume the norms to be
unitarily invariant.
We make use of the Frobenius norm, but will denote it simply as $\| \cdot \|$ without
subscripted F.

The preprocessing step brings the matrix
polynomial's leading and trailing coefficients to triangular
form and all other polynomial coefficients are transformed via a unitary
equivalence $\tilde{P}_i=U^* P_i V$.  In floating point arithmetic, however, we get
$\hat{P}_i$. Relying on the backward stability of the QZ algorithm 
and since both $U$ and $V$ are unitary we get that
$\hat{P}_i=U^*(P_i+\delta P_i)V$, where
$\|\delta P_i\|/\|P_i\| \lesssim u$, where $u$ denotes the unit round-off \cite{b029}. 
Factoring the pencil matrices is free of errors and provides us unitary-plus-spike matrices

In Section~\ref{be:ups} we analyze unitary-plus-spike matrices and see that we end up with
a highly structured backward error. In Sections~\ref{be:gauss} and \ref{be:frobenius}
we consider the combined error of the reduction and eigenvalue computations for the Frobenius and Gaussian factorizations.
We prove and show in the
numerical experiments, Section~\ref{sec:numexp}, that both factorizations have the
error bounded by the same order of magnitude, but the
Gaussian
factorization has a smaller constant.
In Section~\ref{be:summary} we conclude by formulating generic perturbation theorems
pushing the error back on the pencil and on the matrix polynomial.
We show that with an appropriately scaling  we end up with a backward stable algorithm.

\subsection{Perturbation results for unitary-plus-spike matrices}
\label{be:ups}
We first state a generic perturbation theorem for upper triangular unitary-plus-spike
matrices.  This theorem is more detailed compared to our previous error bound \cite{AuMaVaWa15}.
The precise location of the error is essential in properly accumulating
the errors of the various factor matrices in Theorems~\ref{thm:gaussbackwarderror} and
\ref{theo:be:frobenius}. We show that the backward error on an upper triangular
unitary-plus-spike matrix that has undergone some pass-through operations is distributed
non-smoothly. Suppose the spike has zeros in the last $\ell$ spots. Then the associated
upper triangular matrix will have its lower $\ell$ rows perturbed only by a modest
multiple of the machine precision.  The other upper rows will incorporate a larger error
depending on $\|R\|$, and the upper part of the spike absorbs the largest error being proportional to $\|R\|^2$.

\begin{theorem} \label{lem:Rperturbation}
  Let $R$ be an $n\times n$ upper triangular identity-plus-spike matrix\footnote{In the
setting of the paper $n=dk$, but the theorem holds for any $n$. In the remainder we identify $n$ with $dk$.} factored as
  \[
    R = P^T  \mathcal C^* (\mathcal B + e_1 \underline{y}^T)  P = I_n + (x - e_{\ell}) y^T, \qquad
    \norm{C^* e_1} = \norm{\underline{x}} = 1,
  \]
with the notational conventions on $x$, $y$, $\underline{x}$, and $\underline{y}$ and the
rank-one part implicitly encoded in the unitary part as in Section~\ref{sec:storage}.
  Let $U$ and $V$ be unitary matrices representing the action of several pass-through
  operations through $R$ such that the resulting $\tilde R = U^* R V$ is  an upper triangular
  unitary-plus-rank-one matrix. 
If $\hat R$ is the result obtained computing $\tilde R$ in floating point, operating on
the unitary part only (see Section~\ref{sec:oper}) and reconstructing the rank-one part
from the unitary part,  we get 
  \[
    \hat R = U^* (R+\delta R) V= U^* (R_u + \delta R_u +R_o +\delta R_o) V, 
  \]
  where $R_u= P^T  \mathcal C^* \mathcal B  P$ and $R_o= P^T  \mathcal C^*  e_1
  \underline{y}^T  P$, and   $\delta R_u$ stands for the error in the unitary and
  $\delta R_o$ in the rank-one part.  The error in the unitary part $\norm{\delta R_u} \lesssim u$, and the error in the
  rank-one part 
  \[
    \delta R_o \doteq - \delta \rho_r\, x y^T + \delta w_1\, y^T + x\, \delta w_2^T,
  \]
  with $\norm{\delta w_1} \lesssim u$, $\norm{\delta w_2} \lesssim \norm{\underline{y}}  u$,
  and $|\delta \rho_r| \lesssim \norm{\underline{y}} u$. 
\end{theorem}

\begin{proof}
The rank-one part of $R$ is stored implicitly
in the unitary matrices so, in practice, the equivalence
transformation $U^* RV$ only manipulates the unitary part $\CS{C}^*\CS{B}$. So we execute $\underline{U}^*
\CS{C}^*\CS{B} \underline{V}$ where, following our notational convention,
$\underline{U}
=\left[
\begin{smallmatrix}
U & 0 \\
0 & 1
\end{smallmatrix}
\right]$ and
$\underline{V}
=\left[
\begin{smallmatrix}
V & 0 \\
0 & 1
\end{smallmatrix}
\right]$. 
This operation is implemented as
$
\underline{U}^* \CS{C}^* W \; W^*  \CS{B} \underline{V}$, first passing core transformations through
$\CS{B}$ and then through $\CS{C}$.  The resulting upper triangular $\tilde R$ can
therefore be factored as  $$
    \tilde R = P^T \underline{\tilde{R}} P = P^T \tilde{\mathcal C}^* (\tilde{\mathcal B}
    + e_1 \tilde{\underline{y}}^T) P,$$ where
\begin{eqnarray*}
\tilde {\mathcal C}   =  W^* {\mathcal C} \underline{U},  & \qquad &
    \tilde {\mathcal B}  =   W^* {\mathcal B} \underline{V}.
\end{eqnarray*}
Letting $\tilde \rho = (e_{n+1}^T \tilde {\mathcal C}^* e_1)$, the vector $\tilde{y}$
is given by
\begin{eqnarray*}
\underline{\tilde y}^T  & = &  - (e_{n+1}^T \tilde {\mathcal C}^* e_1)^{-1}
  (e_{n+1}^T \tilde {\mathcal C}^* \tilde {\mathcal B})
  = - \tilde{\rho}^{-1}
  (e_{n+1}^T \tilde {\mathcal C}^* \tilde {\mathcal B}).
\end{eqnarray*} 
Since $W e_1 = e_1$ and $\underline{U} e_{n+1}=e_{n+1}$ 
we have that 
$\tilde{\rho}=  e_{n+1}^T {\tilde{\CS{C}}}^* e_1 = e_{n+1}^T {\CS{C}}^* e_1 = \rho$, which in turn is the final element of $\underline{x}$.  
 
In floating point, however, we end up with $\hat{R}$ rebuilt from
the computed $\hat{\mathcal C}$ and $\hat{\mathcal B}$, which are the perturbed version of $\tilde{\mathcal
  C}$ and $\tilde{\mathcal B}$, by using
 $$
    \hat R = P^T \underline{\hat R} P = P^T \hat{\mathcal C}^* (\hat{\mathcal B} + e_1
    \hat{\underline{y}}^T) P, 
  $$
where, for $\hat\rho = (e_{n+1}^T \hat {\mathcal C}^* e_1)$, the vector $\hat{y}$ is
computed from
$$
  \underline{\hat y}^T   =  - (e_{n+1}^T \hat {\mathcal C}^* e_1)^{-1}
  (e_{n+1}^T \hat {\mathcal C}^* \hat {\mathcal B})
=       - \hat{\rho}^{-1}
  (e_{n+1}^T \hat {\mathcal C}^* \hat {\mathcal B}).
$$

The matrices $\hat{\mathcal C}$ and $\hat{\mathcal B}$ are the result of executing
turnovers on $\mathcal C$ and $\mathcal B$.
Each turnover introduces
a small backward error  of the order $u$ \cite{AuMaVaWa15} and
we have, for $\norm{\delta \CS{C}}\lesssim
u$, $\norm{\delta \CS{B}} \lesssim u$, the relations
$    \hat {\mathcal C} = W^* ({\mathcal C} + \delta {\mathcal C}) \underline{U}$ and
$     \hat {\mathcal B} = W^* ({\mathcal B} + \delta {\mathcal B})\underline{V}$. All the individual
perturbations are accumulated in $\delta
{\mathcal C}$ and $\delta{\mathcal B}$, which could be dense, unstructured, in general.
 It's
 worth noting, however, that the implementation ensures that $\hat{\mathcal C}$ as
 well as $\hat{\mathcal B}$ remain unitary.

Let us focus now on the factor $\hat \rho$.
 In floating point we compute $\hat \rho =  \rho + \delta \rho_a =
 \rho (1 + \delta \rho_r)$ (subscripts $a$ and $r$ denoting  the
absolute and relative error respectively).
Given that $|\delta \rho_a| \lesssim u$ we have $|\delta \rho_r| \lesssim  |\rho|^{-1} u$,
which can be bounded by 
 $|\delta {\rho_r}| \lesssim \norm{\underline{y}} u$. Under the assumption that $\delta \rho_r$ is tiny we get
$(1+\delta \rho_r)^{-1}\doteq (1-\delta \rho_r)$.

Combining these relations and using $W e_1 = e_1$ and $\underline{U} e_{n+1}=e_{n+1}$
leads to 
\begin{eqnarray*}
\underline{\hat{R}} & =& 
\underline{U}^* (\CS{C}+\delta \CS{C})^* W\;
\Big(
          W^* (\CS{B}+\delta \CS{B}) \underline{V}
\\ &&\quad + e_1
(-\rho^{-1}) ( 1+\delta\rho_r)^{-1} 
(e_{n+1}^T \underline{U}^* (\CS{C}+\delta \CS{C})^* W \;
          W^* (\CS{B}+\delta \CS{B}) \underline{V})
\Big) \\
& \doteq & \underline{U}^* \Big(
(\CS{C}+\delta\CS{C})^* (\CS{B}+\delta\CS{B})
-
(\CS{C}+\delta{\CS{C}})^* e_1
\rho^{-1} ( 1-\delta\rho_r) 
e_{n+1}^T (\CS{C}+\delta \CS{C})^* 
          (\CS{B}+\delta \CS{B}) 
\Big) \underline{V}.
\end{eqnarray*}
Using $\underline{x}=\CS{C}^* e_1$ and
$\underline{y}^{T}=-\rho^{-1} (e_{n+1}^T \CS{C}^* \CS{B})$ provides the following
approximation of the error
\begin{eqnarray*}
  U (\hat R - \tilde R) V^* &\doteq
  P^T\Big(\delta {\mathcal C}^*\, {\mathcal B} + {\mathcal C}^*\, \delta {\mathcal B}
+ \delta {\mathcal C}^* e_1 \underline{y}^T
- \delta \rho_r \, \underline{x} \underline{y}^T
- \rho^{-1} \underline{x} e_{n+1}^T (\delta {\mathcal C}^*\,\mathcal B + {\mathcal C}^*\, \delta \mathcal B) \Big) P \\
  &= \delta R_u + \delta w_1\, y^T - \delta \rho_r\, xy^T  + x \,\delta w_2^T,
\end{eqnarray*}
where  $\delta R_u = P^T \left(\delta {\mathcal C}^*\, {\mathcal B} + {\mathcal C}^*\, \delta
{\mathcal B}\right) P$. The bounds on the norms of $\delta w_1$ and $\delta w_2$ follow  directly.
\end{proof}

  We can show that
  $\norm{\underline y} \leq \norm{\underline{R}} \leq 1 + \norm{\underline y}$
  and $\norm{R} \leq \norm{\underline R} \leq 1 + \norm{R}$.  Looking at the four terms
  that comprise the backward error, we see that $\|\delta R_{u}\| \lesssim u$, 
  $\norm{\delta w_{1}\,y^{T}} \lesssim \norm{y}u \lesssim \norm{R}u$, 
  $\norm{\delta\rho_{r}xy^{T}} \lesssim \norm{y}^{2}u \lesssim \norm{R}^{2}u$, and 
  $\norm{x\,\delta w_{2}^{T}} \lesssim \norm{y}u \lesssim \norm{R}u$.    
  We can conclude that  $\|\delta R\|\lesssim \|R\|^2 u$. 
  
  Note, however, that only one of the four terms has a $\norm{R}^{2}$ factor, namely
  $\delta\rho_{r}xy^{T}$.  This is a backward error; the rank-one matrix $xy^{T}$ is the initial 
  rank-one matrix, which is in the shape of the initial spike.  
  It follows that the part of the 
  backward error that depends on $\norm{R}^{2}$ is confined to the spike.   The regions
  of dependence can be depicted as follows.  
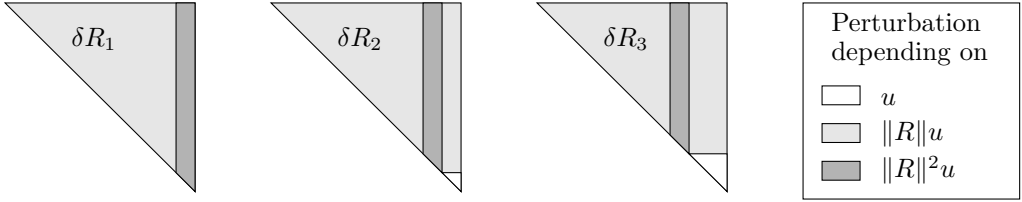
\begin{equation*}
\begin{tikzpicture}[baseline={(current bounding box.center)},scale=.5,y=1cm]
\draw[fill=black!10!white] (0,0) -- (5,0) -- (5,-5) -- cycle;
\draw[fill=black!30!white] (4.5,0)--(5,0) -- (5,-5) -- (4.5,-4.5) --cycle;
\node[above right] at (1.5,-1.5) {$\delta R_{1}$};

\begin{scope}[xshift=7cm]
\draw[fill=black!10!white] (0,0) -- (5,0) -- (5,-5) -- cycle;
\draw[fill=black!30!white] (4.0,0)--(4.5,0) -- (4.5,-4.5) -- (4.0,-4.0) --cycle;
\draw[fill=white] (4.5,-4.5) -- (5,-4.5) -- (5,-5) -- cycle;
\node[above right] at (1.5,-1.5) {$\delta R_{2}$};
\end{scope}

 \begin{scope}[xshift=14cm]
 \draw[fill=black!10!white] (0,0) -- (5,0) -- (5,-5) -- cycle;
 \draw[fill=black!30!white] (3.5,0)--(4,0) -- (4,-4) -- (3.5,-3.5) --cycle;
 \draw[fill=white] (4.0,-4.0) -- (5,-4.0) -- (5,-5) -- cycle;
 \node[above right] at (1.5,-1.5) {$\delta R_{3}$};
 \end{scope}




\begin{scope}[xshift=21cm]
\node[above right] at (0.5,-1.0) {Perturbation};
\node[above right] at (0.5,-1.9) {depending on};
\draw (0,0.0) rectangle (5.7,-5.2);
\draw[fill=white] (0.5,-2.2) rectangle (1.5,-2.7);
\node[above right] at (1.8,-2.95) {$u$};
\draw[fill=black!10!white] (0.5,-3.2) rectangle (1.5,-3.7);
\node[above right] at (1.8,-4.05) {$\|R\| u$};
\draw[fill=black!30!white] (0.5,-4.2) rectangle (1.5,-4.7);
\node[above right] at (1.8,-5.05) {$\|R\|^2 u$};

\end{scope}

\end{tikzpicture}
\end{equation*}

\subsection{Gaussian factorization}
\label{be:gauss}
We first consider the Gaussian factorization, and we write $S$ 
and $T$ as
\[
  S = Q R_1 \cdots R_k , \mbox{ and }
  T = T_1 \cdots T_k. 
\]
The structure of $T$ is a particular case of the one of $S$ where we
have chosen $Q = I_n$. Besides $Q$ both matrices consist of a product of
unitary-plus-spike matrices, and we will use the perturbation results from Theorem~\ref{lem:Rperturbation}
for each of these factors individually.
For simplicity we perform the backward error analysis
on $S$ only. The same holds unchanged for $T$ by ignoring $Q$.

At the end of the algorithm,
when $S$ has reached upper triangular form, we will have
$\tilde S = U^* S V$. In floating point we compute a matrix
$\hat S$ such that $\hat S = U^* (S + \delta S) V$.
In the Gaussian case we have that, for $j=1,\ldots,k$, 
\begin{equation}
\label{R:gf}
  R_j = I_n + (x_j - e_\ell) y_j^T, \qquad
  y_j = \alpha_j e_{\ell},
\end{equation}
$\ell = n-j+1$.   Only the first $\ell$ elements of $x_j$ can be different from zero. We remark that $x_j$ is
normalized such that its enlarged version has norm one: $\|\underline{x}_j\|=1$.
We are interested in bounding
the norm of $\delta S$. 
A simple lemma is required, before stating the main result of the section.
\begin{lemma}
  \label{lem:rxj}
  For $j = 1, \ldots, k$ and $R_j$ the factors
  of the Gaussian factorization \eqref{R:gf}
 we have that $R_1 \cdots R_{j-1} x_j = x_j$ and  $y_j^T R_{j+1} \cdots R_{k} = y_j^T$. 

\end{lemma}
\begin{proof}
Whenever $i < j$, we have $y_{i}^T x_j = \alpha_{i} e_{n-i+1}^T x_j =0$,
because of the zero structure of $x_j$. The relation $R_1 \cdots R_{j-1} x_j = x_j$
follows from \eqref{R:gf}. Similar arguments prove the second relation.

\end{proof}






\begin{theorem}
  \label{thm:gaussbackwarderror}
  Let $\hat S$ be the result of the floating point computation of
  $\tilde S = U^* S V$, where a Gaussian factorization of $S$ was used. Then we have that
  $$\hat S = U^* (S + \delta S) V,\;\; \mbox{with}\;\; \norm{\delta S} \lesssim \norm{S}^2 u.$$ 
\end{theorem}

\begin{proof}
  The upper triangular matrix $\tilde S$ is expressed as
  the product of upper triangular matrices $\tilde R_j=U^*_{j-1} R_j U_j$, for
$j=1,\ldots,k$ (set $U_k=V$), with the matrices $R_j$ as in \eqref{R:gf}:
  \[
    \tilde S = U^* S V =
    \underbrace{U^* Q U_0}_{\tilde{Q}=I_n}\,
    \underbrace{U_0^* R_1 U_1}_{\tilde R_1}\,
    \underbrace{U_1^* R_2 U_2}_{\tilde R_2}\,
    \cdots\,
    \underbrace{U_{k-1}^* R_k V}_{\tilde R_k}. 
  \]
  
Applying Theorem~\ref{lem:Rperturbation} to each of the upper triangular factors
gives $\hat{R}_j= U^*_{j-1} (R_j+\delta R_j) U_j$. 
For the unitary part we have 
$U^*(Q + \delta Q)U_0$, with $\|\delta Q\| \lesssim u$.
Combining the relations for the upper triangular and unitary part yields, up to first order terms: 
\begin{equation}
\label{eq:cerror}
  U \hat S V^* - S \doteq \delta Q\, R_1 \cdots R_k + \sum_{j = 1}^k Q\, R_1 \cdots R_{j-1}\, \delta R_j\,
  R_{j+1} \cdots R_k. 
\end{equation}

Relying on the explicit form of the perturbation from Theorem~\ref{lem:Rperturbation} 
combined with Lemma~\ref{lem:rxj} 
allows to rewrite the product of the upper triangular factors as follows
  \begin{eqnarray*}
    R_1 \cdots R_{j-1} \,\delta R_j\, R_{j+1} \cdots R_k &\doteq
    R_1 \cdots R_{j-1}\, \delta R_{j,u} \,R_{j+1} \cdots R_k -  \delta \rho_{j,r}\, x_j y_j^T \\
    &+ R_1 \cdots R_{j-1} \,\delta w_{j,1}\, y_j^T
    + x_j\, \delta w_{j,2}^T\, R_{j+1} \cdots R_k. 
  \end{eqnarray*}
  Noting that $\norm{R_{1} \cdots R_{j-1}} \leq \norm{S}$, $\norm{R_{j+1} \cdots R_{k}} \leq \norm{S}$, 
and $\norm{\underline{y_j}} \lesssim \norm{S}$,  we can bound each of the above
  terms by a modest multiple of $\norm{S}^2 u$.
Plugging this bound into \eqref{eq:cerror} and taking into
account that $\|\delta Q\, R_1 \cdots R_k\| \lesssim \|S\| u$ gives the result.
\end{proof}
As a consequence we immediately know also that the backward error on $T$ is bounded by a modest multiple of $\|T\|^2 u$.

\subsection{Frobenius factorization}
\label{be:frobenius}
The Frobenius factorization is slightly more difficult to analyze.
At the start, we factor $S$ as
\[
  S = \mathcal Q R_1 \cdots \mathcal Q R_k = S_1 \cdots S_k.
\]
We are able to prove roughly the same results, that means an error depending on $\|S\|^2$
only, but we will see that the constant could be much larger than in the Gaussian case,
due to the intermediate factor matrices $Q$. In the Gaussian case there is only a single
factor in front the unitary-plus-spike matrices; here, however, the interlacing unitary
factors will spread out the errors more leading to a higher constant.

 According to Section~\ref{sec:mmp} each $S_j$ is of the
form \begin{equation}
\label{eq:ff}
S_j = \mathcal Q (I_n + x_j y_j^T),\;\; \mbox{ with }\;\;
  y_j = \alpha_j e_{n},
\end{equation}
and $\CS{Q} x_j$ can only have its first $n-j+1$ elements different from zero.

\begin{lemma}
  \label{lem:xjRcomp}
  For $j = 1, \ldots, k$ and $S_j$ the factors \eqref{eq:ff} of the Frobenius
  factorization we have
  $S_{1} \cdots S_{j-1} \CS{Q} x_j = \CS{Q}^{j} x_j$ and
  $y_j^T S_{j+1} \cdots S_{k} = y_j^T \mathcal Q^{k-j}$.
\end{lemma}
\begin{proof}
We know that each $\CS{Q} x_j$ has $j-1$ trailing zeros (for $1\leq j \leq k$), as a consequence, recall the structure
\eqref{matrixQ} of $\CS{Q}$, we have that $\CS{Q}^{j-\ell}x_j=\CS{Q}^{j-\ell-1} \CS{Q} x_j$ has $\ell$ trailing zeros.
As a consequence we have that $y_\ell^T \CS{Q}^{j-\ell} x_j =\alpha_\ell e_{n}^T
\CS{Q}^{j-\ell} x_j=0$ as long as $\ell <j $.
 Using \eqref{eq:ff}
proves the first relation, the second relation is obtained similarly.
\end{proof}

\begin{theorem}
\label{theo:be:frobenius}
  Let $\hat S$ be the result of the floating point computation of
  $\tilde S = U S V^*$, where a Frobenius factorization of $S$ was used. Then we have that
  $\hat S = U (S + \delta S) V^*$ with
  $\norm{\delta S} \lesssim  \norm{S}^2 u$.   
\end{theorem}

\begin{proof}
  The upper triangular matrix $\tilde S$ can be expressed as
  the product of upper triangular matrices $\tilde R_j$ and unitary matrices $\tilde{Q}_i=I_n$ as follows:
  \[
    \tilde S = U^* S V =
    \underbrace{U^* \mathcal Q X_0}_{\tilde Q_1}\,
    \underbrace{X_0^* R_1 U_1}_{\tilde R_1}\, \underbrace{U_1^* \mathcal Q X_1}_{\tilde Q_2}\,
    \underbrace{X_1^* R_2 U_2}_{\tilde R_2}\,
    \cdots
    \underbrace{U_{k-1} \mathcal Q X_{k-1}^*}_{\tilde Q_{k}}\,
    \underbrace{X_{k-1}^* R_k V}_{\tilde R_k}. 
  \]
  Using again Theorem~\ref{lem:Rperturbation} to bound the perturbation on the upper
triangular factors and assuming that 
in floating
  point we have computed $U_j^*(\mathcal Q + \delta \CS{Q}_j) X_j$, for $j=0,\ldots,k-1$
(set $U_0=U$), with $\|\delta
Q_j\|\lesssim u$,
  yields, up to first order terms: 
  \[
    U \hat S V^* - S \doteq
    \sum_{j = 1}^k S_1 \cdots S_{j-1}\, \delta \CS{Q}_j \,R_j \,S_{j+1} \cdots S_k +
    \sum_{j = 1}^k S_1 \cdots S_{j-1}\, \mathcal Q \,\delta R_j\,
    S_{j+1} \cdots S_k. 
\]
For the left part of the summand we get $\|S_1 \cdots S_{j-1}\, \delta \CS{Q}_j \,R_j
\,S_{j+1} \cdots S_k \| \lesssim \|S\|^2 u$. The right part 
can be analyzed similarly as before and by 
Theorem~\ref{lem:Rperturbation} and Lemma~\ref{lem:xjRcomp} 
we get
  \begin{eqnarray*}
    S_1 \cdots S_{j-1}\, \mathcal Q \, \delta R_j \, S_{j+1} \cdots S_k &\doteq
    S_1 \cdots S_{j-1}\, \mathcal Q \, \delta R_{j,u} S_{j+1} \cdots S_k -
    \delta \rho_{j,r}\, \CS{Q}^j x_j y_j^T \CS{Q}^{k-j} \\
    &+ S_1 \cdots S_{j-1} \mathcal Q\, \delta w_{j,1}\, y_j^T \CS{Q}^{k-j}
    + \CS{Q}^j x_j \,\delta w_{j,2}^T\, S_{j+1} \cdots S_k. 
  \end{eqnarray*}
  All of the terms above can be bounded by $\norm{S}^2 u$ and the result follows.
\end{proof}

\begin{remark}
Both factorizations yield a quadratic dependency on the norm of $S$ for
the backward error. A careful look at the proof reveals, however, that the bound in the
Gaussian factorization is smaller than the one of the Frobenius case. In the latter
case we have $k$ terms  $S_1 \cdots S_{j-1}\, \delta \CS{Q}_j \,R_j
\,S_{j+1} \cdots S_k$ contributing to the quadratic error in $\|S\|$, and these terms are not present
in the Gaussian factorization.
\end{remark}

The error bounds can be improved significantly by scaling the problem. It is
straightforward to apply a diagonal scaling from the right on the matrix polynomial such that the
entire block column in \eqref{eq:comp_pencil} has norm $1$. This scaling is used in the numerical experiments.

\subsection{Main backward error results}
\label{be:summary}
The generic perturbation results  can be summarized into three theorems.

\begin{theorem}[Backward error on the block companion pencil]
\label{be:th1}
Given a block companion pencil $S-\lambda T$, the Schur form computed via the
algorithm proposed in this paper is the exact Schur form of a perturbed pencil $(S+\delta
S)- \lambda (T+\delta T)$, where $\|\delta T\|\lesssim \|T\|^2 u$, $\|\delta S\|\lesssim
\|S\|^2 u$, and $u$ is the unit round-off.
\end{theorem}
The proof is a combination of the results of this section.

\begin{theorem}[Backward error on the block companion pencil in case of scaling]
  \label{thm:backward_error_scaling}
  Given a block companion pencil $S-\lambda T$ and run the algorithm proposed in this paper 
  on the scaled pencil $(\alpha^{-1} S) - \lambda (\alpha^{-1} T)$, where 
  $\alpha=\max ( \|T\|, \|S\|)$.  Then the
  computed Schur form is the exact Schur form of a perturbed pencil $(S+\delta S)- \lambda
  (T+\delta T)$, where $\|\delta T\|\lesssim \alpha\, u$, $\|\delta S\|\lesssim \alpha\, u$,
  and $u$ is the unit round-off.
\end{theorem}

\begin{proof}
 Applying Theorem~\ref{be:th1} to
$(\alpha^{-1} S) -\lambda (\alpha^{-1} T)$, we get the following bounds on the backward errors 
$\norm{\delta (\alpha^{-1} S)} \lesssim \norm{\alpha^{-1}S}^{2}u \leq u$
and $\norm{\delta (\alpha^{-1} T)} \lesssim \norm{\alpha^{-1}T}^{2}u \leq u$
of the perturbed pencil $(\alpha^{-1} S + \delta(\alpha^{-1} S) )- \lambda ( \alpha^{-1} +
\delta(\alpha^{-1}T))$ whose Schur form we have actually computed.

Mapping this to the original pencil $(S+\delta S)- \lambda
  (T+\delta T) = (S+\alpha\, \delta(\alpha^{-1} S))- \lambda
  (T+\alpha\, \delta(\alpha^{-1} T))$
proves the theorem.
\end{proof}

For our final result we make the assumption that 
$$\sqrt{\| [M_0,\ldots,M_{d-1}]\|^2+\|[N_1,\ldots,N_d]\|^2}
\approx\sqrt{\sum_{i=0}^d \| P_i \|^2 }.$$ 

\begin{theorem}[Backward error on the matrix polynomial]
  \label{thm:backward_error_mpoly}
  Given a matrix polynomial $P(\lambda)=\sum_{i=0}^d P_i \lambda^i,$
  whose eigenvalues are computed via the algorithm in this paper,
  including a scaling of the order $ \sqrt{\sum_{i = 0}^d \norm{P_i}^2}$. Then we know that these eigenvalues are
  the exact eigenvalues of a nearby polynomial
  \[
      P(\lambda)+\delta P(\lambda) = \sum_{i=0}^d (P_i(\lambda)+\delta P_i(\lambda))
\lambda^i,\] where
\[\sqrt{\sum_{i=0}^d \| \delta P_i \|^2 } 
\lesssim \,u\,\sqrt{\sum_{i=0}^d \|  P_i \|^2 } .\]
\end{theorem}

\begin{proof}
  The eigenvalues and eigenvectors of a matrix polynomial are
  invariant with respect to scaling, so we compute the eigenvalues
  of the scaled matrix polynomial
  \[
    Q(\x) = \alpha^{-1} P(\x),\; \mbox{ where } \;
    \alpha \approx \sqrt{\sum_{i = 0}^d \norm{P_i}^2}.
  \]
  According to Theorem~\ref{thm:backward_error_scaling} this provides
  the Schur form of the block companion pencil of $Q(\x)$ computed
  with a backward error of the order of the machine precision
  $u$.  Both the block companion pencil and the
  gathered coefficients of $Q(\x)$ have norm about $1$. Relying on 
the work of Edelman and Murakami \cite{edelman1995polynomial},
  or Dopico, Lawrence, P\'erez, and Van Dooren \cite{DoLaPeVD16}
this implies that we
  have computed the exact eigenvalues of a nearby polynomial
  $Q(\x) + \delta Q(\x)$, with
  $\sqrt{ \sum_{i = 0}^d \norm{\delta Q_i}^2} \lesssim
  u$.
  Therefore, we have computed the exact eigenvalues of
  $P(\x) + \delta P(\x)$, with $\delta P(\x) = \alpha\, \delta Q(\x)$,
  and $\delta P(\x)$ satisfies
  \[
    \sqrt{\sum_{i = 0}^d \norm{\delta P_i}^2} =
    \alpha \sqrt{\sum_{i = 0}^d \norm{\delta Q_i}^2} \lesssim
    \,u\,\sqrt{\sum_{i = 0}^d \norm{P_i}^2} . 
  \]
\end{proof}

We remark that the results in this paper provide normwise backward error results on all
coefficients simultaneously and coefficient-wise backward stability cannot be guaranteed
by these theorems.

So far we have not yet discussed the computation of the eigenvectors. The only operations
involved are computing the swapping core transformations, turnovers to move the swapping
core transformations through the upper triangular factors, and updating the matrix $PU$.
These are all backward stable and as a result retrieving a single eigenvector is backward
stable.
It must be said,
however, that we can only state that a single eigenpair is computed stably. Stability does not
necessarily hold for the entire eigendecomposition as the perturbation will be eigenvector
dependent.   In other words, we do not claim that
there is a single small perturbation on the matrix polynomial such that its exact eigenvectors
match all our computed eigenvectors.

\section{Numerical experiments}
\label{sec:numexp}
The algorithm is implemented in the software package {\tt eiscor}, which provides eigenvalue
algorithms based on core transformations. The software can be freely downloaded 
from Github by visiting {\tt https://github.com/eiscor/eiscor/}. 
We examine the computational complexity of computing only eigenvalues, the backward error
on the Schur form, complexity and stability of computing eigenpairs, and we conclude with
some examples from the NLEVP collection.

  \subsection{Computing eigenvalues: complexity analysis}
  
  We verify the asymptotic computational complexity of the method. We proved in
  Section~\ref{sec:reduction} that  
  the expected computational complexity is  $\mathcal O(d^2 k^3)$.
Two tests were executed.
  
  \begin{itemize}
  \item We verified the quadratic complexity in $d$ by fixing $k = 4$ and then computing
    eigenvalues of random matrix polynomial eigenvalue problems for different values of
    $d$.  We compared the timings with the QZ iteration implemented in LAPACK 3.6.0.
      
    \item We verified the cubic complexity in $k$ by fixing $d = 4$
      and running the algorithm for different values of $k$ ranging
      between $1$ and $1024$. 

  \end{itemize}
  
  In Figure~\ref{fig:quadratic}, fixing $k$,  we notice that the current implementation is faster
  than LAPACK at about $d = 40$.   In Figure~\ref{fig:cubic}
  we   plotted the complexity 
  as a function of the size
  of the coefficient matrices. The  plot shows 
  that the slope of the curve representing the reduction
   is well approximated by $3$, indicating a cubic dependency on $k$.  

  
  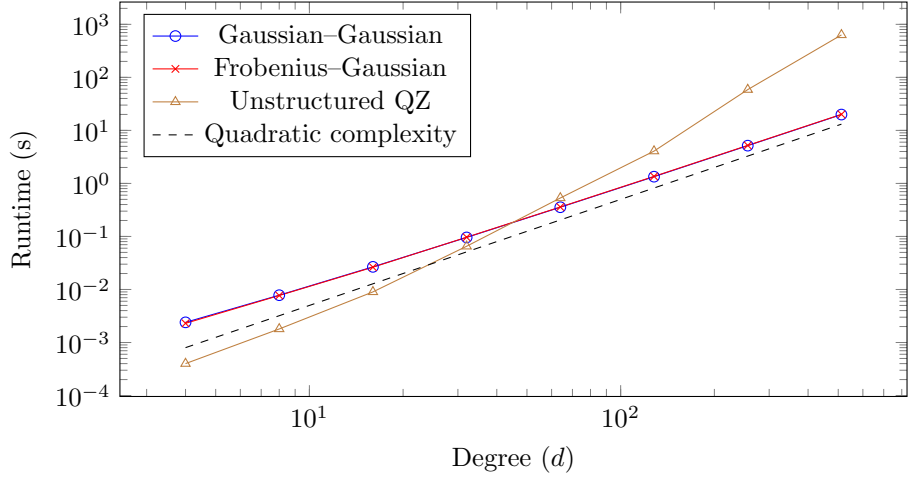
\begin{figure}
  \begin{center}
    \begin{tikzpicture}
    \begin{loglogaxis}[
      xlabel=Degree ($d$),ylabel=Runtime (s),
      width=.8\linewidth,
      height=.3\textheight,
      legend pos = north west]
      \addplot[mark=o,color=blue]
         table[x=Degree,y=Eiscorupperlower] {complexity.dat};
      \addplot[mark=x,color=red]
         table[x=Degree,y=Eiscorupper] {complexity.dat};    
      \addplot[mark=triangle,color=brown]
         table[x=Degree,y={Polyeig}] {complexity.dat};
      \addplot[dashed,domain=4:512] {5e-5 * x^2};
    \legend{Gaussian--Gaussian, Frobenius--Gaussian, Unstructured QZ, Quadratic complexity}; 
    \end{loglogaxis}
    \end{tikzpicture}
    
   \end{center}
  
   \caption{Test of the quadratic complexity in the degree $d$ of the matrix polynomial,
     $k=4$, and the tests were averaged over $10$ runs. The runtime is compared with the one
     of the unstructured QZ algorithm from LAPACK. The dashed line represents a quadratic
     dependence of the runtime on the degree and is added for reference. The runtimes are
     reported for (Frobenius,Gaussian)-factored and (Gaussian,Gaussian)-factored pencils,
     whose timings are almost indistinguishable.}
    \label{fig:quadratic}
  \end{figure}
  
  
  \begin{figure}
    \centering 
  \begin{tikzpicture}
    \begin{loglogaxis}[
      width=.8\linewidth,
      height=.3\textheight,
      xlabel=Size ($k$),ylabel=Runtime (s),
      legend pos = north west,
      ymax = 1e7]
    \addplot table[x=Size,y=fast] {complexity_size.dat};1
    \addplot[dashed,domain=2:512] {1e-4 * x^3}; 
    \legend{Reduction, $y = 10^{-4} x^3$}; 
    \end{loglogaxis}
  \end{tikzpicture}
    \caption{Test of the complexity in the size of the matrices $k$.  The examples all
      have  degree $d = 4$ and for each combination of $d$ and $k$ $10$ tests were run.
}
            \label{fig:cubic}

  \end{figure}
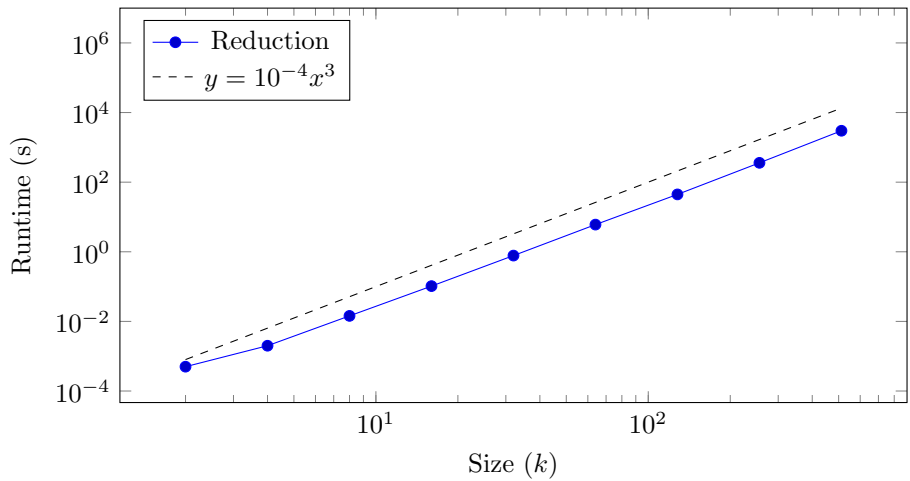

  \subsection{Backward stability of the Schur form}
  
  In Section~\ref{sec:backward} we provided bounds on the 
  backward error of the computed Schur form.
In order to validate these bounds we have measured the backward error on the computed
  upper triangular pencil $\hat{S} - \lambda \hat{T}$ by evaluating
  $\lVert U \hat{S} V^* - S \rVert_F$ and $\lVert U \hat{T} V^* - T \rVert_F$, 
  where $S - \lambda T$ is the companion pencil associated 
  to a matrix polynomial with random coefficients. We have run
  1000 experiments for $k=8$ and $d=10$.
  The results are reported in Figure~\ref{fig:backcloud}
for a (Frobenius,Gaussian)-factored pencil and in 
Figure~\ref{fig:backclouddavid} for a (Gaussian,Gaussian)-factored pencil.
We see that, even though both approaches exhibit a quadratic growth in $S$, the
(Gaussian,Gaussian)-factorization provides, for this test setting, the best backward error.
  Both plots also show that the bounds that we have found are asymptotically
  tight.
  \begin{figure}
    \centering
    \tikzsetnextfilename{errA}
    \begin{tikzpicture}
      \begin{loglogaxis}[
        width=.8\linewidth,
        height = .3\textheight,
        legend pos = north west,
        ylabel={$\norm{S-V\hat{S}W^{H}}_F$, $\norm{T-V\hat{T}W^{H}}_F$},%
        xlabel={$\norm{S}_F, \norm{T}_F$},
        ]
        \addplot[only marks, red, mark=*] table {errA.dat};
        \addplot[only marks, green, mark=triangle] table {errB.dat};
        \addplot[domain=2:1e7, dashed] {2.22e-15 * x^2};
        \legend{Backward error on $S$, Backward error on $T$, {$\approx \norm{\cdot}_F^2 \cdot u$}}
      \end{loglogaxis}
    \end{tikzpicture}
    
    \caption{Backward error on the computed Schur form for different
      values of $\norm{S}_F$ and $\norm{T}_F$. We took $k = 8$, $d= 10$ and ran 1000
      tests. For $S$ a Frobenius factorization and for $T$ a Gaussian factorization was used.
      The dashed lines represent a reference line for the quadratic complexity.
      }
      \label{fig:backcloud}
    \end{figure}
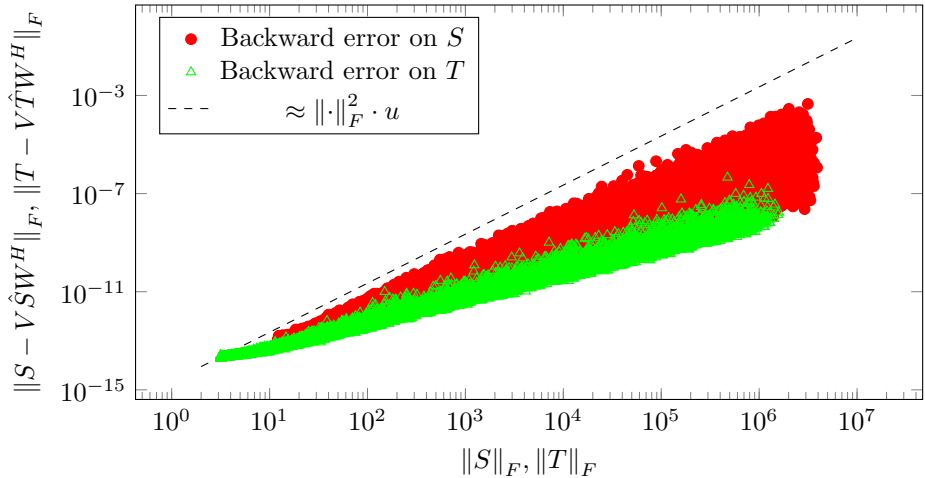

      \begin{figure}
        \centering       
        \tikzsetnextfilename{errAdavid}
    \begin{tikzpicture}
      \begin{loglogaxis}[
        width=.8\linewidth,
        height = .3\textheight,
        legend pos = north west,
        ylabel={$\norm{S-V\hat{S}W^{H}}_F$, $\norm{T-V\hat{T}W^{H}}_F$},%
        xlabel={$\norm{S}_F, \norm{T}_F$},
        ]
        \addplot[only marks, red, mark=*] table {errAdavid.dat};
        \addplot[only marks, green, mark=triangle] table {errBdavid.dat};
        \addplot[domain=2:1e7, dashed] {2.22e-15 * x^2};
        \legend{Backward error on $S$, Backward error on $T$, {$\approx \norm{\cdot}^2_F \cdot u$}}
      \end{loglogaxis}
    \end{tikzpicture}
    
    \caption{Backward error on the computed Schur form for different
      values of $\norm{S}_F$ and $\norm{T}_F$. We took $k = 8$, $d= 10$ and 1000
      runs. In this example
      the Gaussian factorization has been used  for both matrices $S$ and $T$.}
      \label{fig:backclouddavid}
  \end{figure}
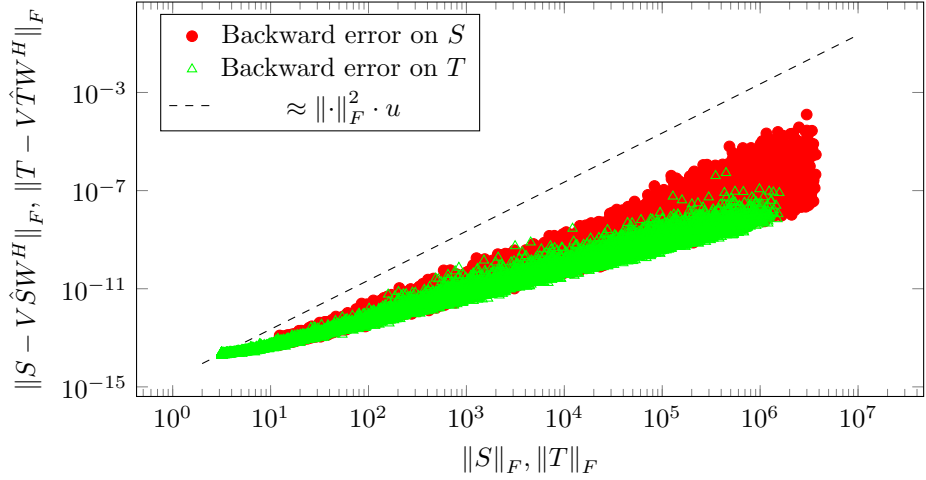

  \subsection{Computing eigenvectors: complexity and stability}

  Finally, we have computed the eigenvalues and the eigenvectors of
  random matrix polynomials of degree $d = 4$ and size $k = 8$.  We
  have generated the coefficients  drawing the entries from a
  normal distribution, and we have randomly scaled each coefficient in
  order to make them of unbalanced
  norms. More precisely, each coefficient is of the form
  $2^{\alpha}  M$, where the entries of $M$ are distributed as Gaussians
  with mean $0$ and variance $1$, and $\alpha$ is drawn from
  the uniform distribution on $[-15, 15]$. 

  According to Tisseur \cite{q654}, the absolute backward error
  on a computed eigenpair $(\lambda, v)$ can be evaluated as
  \begin{equation} \label{eq:mpoly_backward_formula}
    err(P, \lambda, v) = \norm{P(\lambda) v} \cdot \left( \sum_{j = 0}^d |\lambda|^j \right)^{-1}.
  \end{equation}

  In Figure~\ref{fig:backwardmpoly} we report the maximum of the
  absolute backward errors on the eigenpairs of $P(\x)$ computed with
  our algorithm; on the $x$ axis we have reported the norm of the
  coefficients of $P(\x)$, computed as the Frobenius norm of
  $[ P_0, \ldots, P_{d} ]$. The linear dependence of the backward
  error on the norm of the coefficients as predicted by Theorem~\ref{thm:backward_error_mpoly} is
  clearly visible.

  \begin{figure}
    \centering
    \begin{tikzpicture}
      \begin{loglogaxis}[
        legend pos = north west,
        width=.9\linewidth,
        height = .3\textheight,
        xlabel = $\norm{[ P_0, \ldots, P_d }_F$, ylabel = Backward error]
        \addplot[only marks, red] table[x index=0, y index=1] {mpoly_be.dat};
        \addplot[domain=5e-3:1e6, dashed] { 1e-14 * x };
        \legend{Backward error, $\mathcal O(\norm{P}) \cdot u$};
      \end{loglogaxis}
    \end{tikzpicture}
    \caption{Asolute backward error on the computed eigenpairs of random
      matrix polynomials $P(\x)$ of different norms, according
      to the formula~(\ref{eq:mpoly_backward_formula}).}
    \label{fig:backwardmpoly}
  \end{figure}
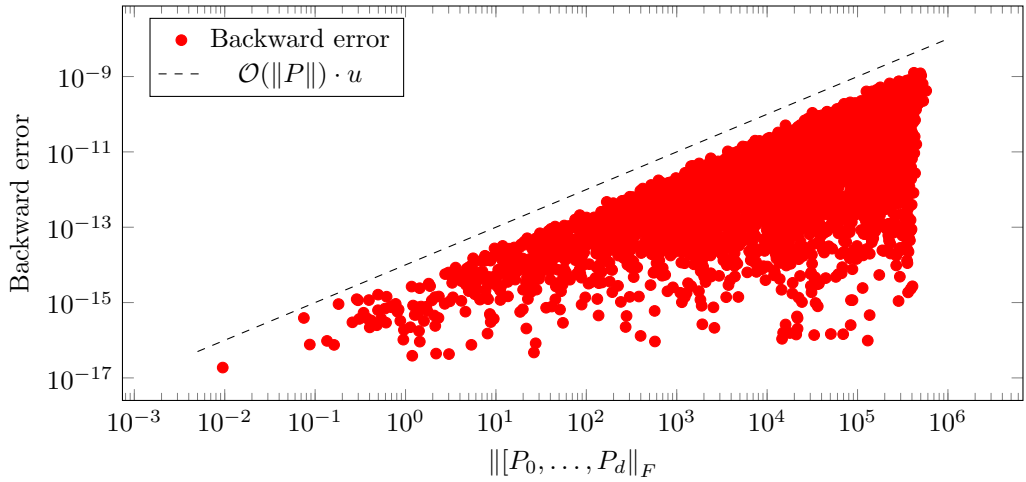

  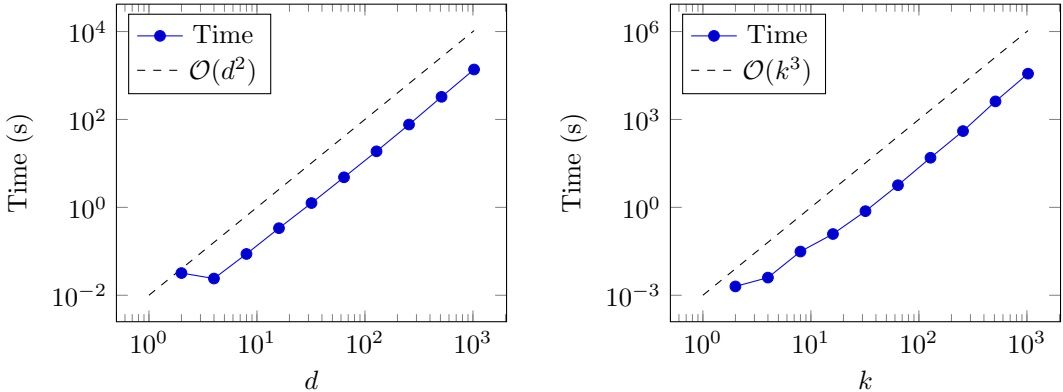
\begin{figure}
    \centering
    \begin{tikzpicture}
      \begin{loglogaxis}[
        legend pos = north west,
        width=.45\linewidth,
        xlabel = $d$, ylabel = Time (s)]
        \addplot table[x index=1, y index=2] {mpoly_time_d.dat};
        \addplot[domain=1:1024, dashed] { 1e-2 * x^2 };
        \legend{Time, $\mathcal O(d^2)$};
      \end{loglogaxis}
    \end{tikzpicture}~~~~~\begin{tikzpicture}
      \begin{loglogaxis}[
        legend pos = north west,
        width=.45\linewidth,
        xlabel = $k$, ylabel = Time (s)]
        \addplot table[x index=0, y index=2] {mpoly_time_k.dat};
        \addplot[domain=1:1024, dashed] { 1e-3 * x^3 };
        \legend{Time, $\mathcal O(k^3)$};
      \end{loglogaxis}
    \end{tikzpicture}
    \caption{Timings for the computation of all the eigenvectors and
      eigenvalues of a matrix polynomial as a function of the degree
      and of the size. The expected quadratic and cubic growth of the
      complexity are visible.}
    \label{fig:timings_mpoly}
  \end{figure}

  In Figure~\ref{fig:timings_mpoly} we have reported the timings for the
  computation of all the eigenvectors of al matrix polynomial $P(\x)$ for
  various degrees and sizes. The numerical results show that the behavior
  remains quadratic in $d$ and cubic in $k$ even when the additional work
  for the computation of the eigenvectors is required. 

  \subsection{Nonlinear eigenvalue problems: NLEVP}
  
  To verify the reliability of our approach we have tested our algorithm on 
  some problems from the NLEVP collection \cite{nlevp}. This archive contains 
realistic polynomial eigenvalue problems. 
  Most of them do have low degree, however, so we have tested our approach only on 
  problems of degree $4$, namely the {\tt orr\_sommerfeld}, the {\tt plasma\_drift}, 
 and the {\tt planar\_waveguide} problems. 
  
  To verify the backward stability, we have
  reported the backward error on the computed
  Schur form in Table~\ref{tab:nlevp-backward}. The pencil was scaled as  in Theorem~\ref{thm:backward_error_scaling}.

  In
  Figure~\ref{fig:backwardeig} we have reported the backward error on
  the single computed eigenpairs according to (\ref{eq:mpoly_backward_formula}).  The
  results clearly show that the proposed method performs as well as the classical QZ method.

  \begin{table}
    \centering
    \begin{tabular}[c]{c|cccc}
      & $\norm{\delta S}$ (eiscor) & $\norm{\delta S}$ (polyeig) & $\norm{\delta T}$ (eiscor) & $\norm{\delta T}$ (polyeig) \\ \hline
      {\tt planar\_waveguide} & 8.361\,e-14 & 8.336\,e-14 & 8.910\,e-14 & 7.214\,e-14 \\
      {\tt orr\_sommerfeld}   & 3.855\,e-14 & 5.523\,e-14 & 3.285\,e-14 & 4.475\,e-14 \\
      {\tt plasma\_drift}     & 6.177\,e-14 & 6.418\,e-14 & 5.705\,e-14 & 4.882\,e-14 \\
    \end{tabular}

    \vspace{12pt}
    
    \caption{Backward error on the Schur form of the properly scaled companion pencil
      for some NLEVP problems. The norms reported are the norms of the perturbations
      to $S$ and $T$.}
    \label{tab:nlevp-backward}
  \end{table}

  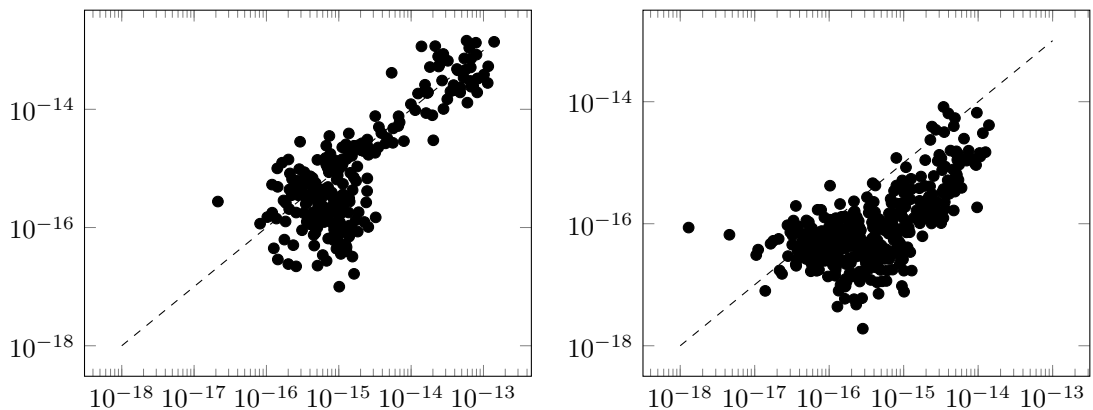
\begin{figure}[ht]
    \centering
    \begin{tikzpicture}
      \begin{loglogaxis}[width=.50\linewidth]
        \addplot[only marks] table[x index=0, y index=1]
        {nlevp_backward_pencil_orr_sommerfeld.dat};
        \addplot[domain=1e-18:1e-13,dashed] { x };
      \end{loglogaxis}
    \end{tikzpicture}~~~\begin{tikzpicture}
      \begin{loglogaxis}[width=.50\linewidth]
      \addplot[only marks] table[x index=0, y index=1]
      {nlevp_backward_pencil_plasma_drift.dat};
      \addplot[domain=1e-18:1e-13,dashed] { x };
      \end{loglogaxis}
    \end{tikzpicture}
    \caption{Backward errors on the computed eigenvalues for the
      {\tt orr\_sommerfeld} (on the left) and
      {\tt plasma\_drift} (on the right) NLEVP
      problems. Each point in the plot has the backward
      error obtained
      using the QZ algorithm on the (scaled) companion
      pencil from LAPACK as y coordinate
      and the algorithm presented in this paper as the x one. }
    \label{fig:backwardeig}
  \end{figure}

\section{Conclusions}
A fast, backward stable algorithm was proposed to compute the eigenvalues of matrix
polynomials. A factorization of the pencil matrices allowed us to design a product eigenvalue
problem operating on a structured factorization of the involved unitary-plus-rank-one
factors. Stability was proved and confirmed by the numerical experiments.

\section*{Acknowledgements}
The authors would like to express their gratitude to the referees, whose careful reading,
error detection, and questions led to a significantly improved version of this paper.

\bibliographystyle{siam}
\bibliography{longstrings,Paper}
\end{document}